\documentclass[journal,10pt]{IEEEtran}  
\IEEEoverridecommandlockouts
\usepackage{amsmath,amssymb,bbm,graphicx}
\usepackage[utf8]{inputenc,xcolor}
\usepackage{subfigure,url}
\usepackage{amsthm}
\newtheorem{definition}{\bfseries Definition}
\newtheorem{proposition}{\bfseries Proposition}
\newtheorem{example}{\bfseries Example}
\newtheorem{theorem}{\bfseries Theorem}

\newtheorem{corollary}{\bfseries Corollary}
\newtheorem{lemma}{\bfseries Lemma}

\newtheorem{remark}{\bfseries Remark}

\allowdisplaybreaks[4]
\newcommand{\vect}[1]{\boldsymbol{\mathbf{#1}}}
\newcommand{\mat}[1]{\begin{bmatrix} #1 \end{bmatrix}}

\newcommand{\R}{\mathbb{R}}
\newcommand{\Z}{\mathbb{Z}}

\newcommand{\la}{\lambda}
\newcommand*\diff{\mathop{}\!\mathrm{d}}
\newcommand{\pa}{\partial}

\newcommand{\lv}{\left\Vert}
\newcommand{\rv}{\right\Vert}
\newif\ifdraft
\drafttrue

\title{Flatness-based Quadcopter Trajectory Planning and Tracking with Continuous-time Safety Guarantees}

\author{Victor Freire,  Xiangru Xu\thanks{Victor Freire and Xiangru Xu are with the Department of Mechanical Engineering, University of Wisconsin-Madison,
        Madison, WI, USA. Email: 
        {\tt\small \{victor.freiremelgizo,xiangru.xu\}@wisc.edu}}}

\begin{document}
\maketitle

\begin{abstract}
This work presents a convex optimization framework for the planning and tracking of quadcopter trajectories with continuous-time safety guarantees. Using B-spline basis functions and the differential flatness property of quadcopters, a second-order cone program is formulated to generate optimal trajectories that respect safe state and input constraints in the continuous-time sense. A quadratic program (QP) based on control barrier functions is proposed to guarantee bounded trajectory tracking in continuous time by filtering a nominal controller, where the QP is shown to be always feasible. Furthermore, conditions that ensure the safe tracking controller respects thrust, roll angle, and pitch angle constraints are also proposed. The effectiveness of the proposed framework is demonstrated by real-world experiments using a Crazyflie2.1 nano quadcopter. 

The video for the experiments of Section \ref{sec:exper} is available at 
\url{https://xu.me.wisc.edu/wp-content/uploads/sites/1196/2021/10/continuous-safety.mp4}.

\end{abstract}

\section{Introduction}

Autonomous quadcopters have been widely used in many fields such as cinematography \cite{cinema}, search and rescue \cite{ludovic2014rescue}, agriculture \cite{navia2016multispectral} and delivery \cite{tang2018aggressive}. Many of these applications are safety-critical and require rigorous satisfaction of safety constraints on the states and inputs of the quadcopters for all time. Although numerous trajectory planning/tracking methods have been proposed for quadcopters, a systematic approach that provides a formal safety guarantee in the continuous-time sense is still largely lacking.

The trajectory planning problem is usually split into finding discrete waypoints that avoid obstacles, and generating a smooth curve that respects the system dynamics through the waypoints. However, a trajectory generated this way is normally only safe at the discrete waypoints without a formal safety guarantee in-between. The inter-sampling safety can be improved by increasing the number of samples at the expense of computation, which works well in practice but lacks  a formal continuous-time guarantee \cite{Dontchev_2019,min_snap,polynomial}. Other works addressing the gap between discrete and continuous-time system trajectories include \cite{dueri_trajectory_2017,bonnans2017error,dontchev1981error}; for example, \cite{dueri_trajectory_2017} predicts when safety will be violated during each inter-sample time interval, but it relies on finite violations derived from polynomial dynamics systems.  Another body of literature exploits the differential flatness property of quadcopters and the convexity property of B-spline basis functions to generate formally safe trajectories in continuous-time \cite{Constrained_trajectory,Flat_trajectory_design,stoical2016obstacle,tang2021real}. These works find a sufficiently smooth trajectory in the flat output space which can be mapped to the states and inputs of the system with algebraic transformations. Recent works along this line of research provide continuous-time safety guarantees in the form of state, thrust and angle constraints \cite{Constrained_trajectory,Flat_trajectory_design} and obstacle avoidance \cite{stoical2016obstacle,tang2021real} constraints.  The main challenge with the flatness-based method lies in the highly nonlinear mapping between state/input spaces and flat output space, which makes the consideration of safety constraints in the state/input space challenging. Furthermore, the resulting optimization problem is usually nonconvex, which renders the online trajectory re-planning computationally intractable.


\begin{figure}
    \centering
    \includegraphics[width=8.5cm]{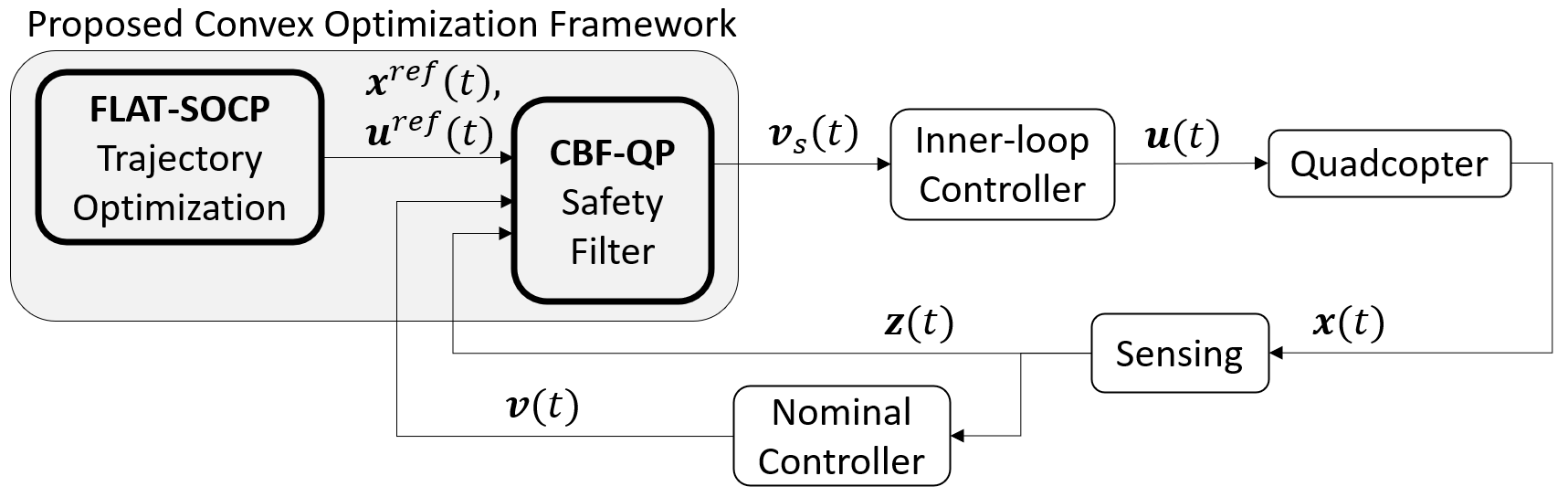}
    \caption{Configuration of the proposed convex optimization-based control framework for quadcopters with continuous-time safety guarantees. The safe trajectory planning is based on solving a second-order cone program, and the safe trajectory tracking is based on solving a quadratic program.}
    \label{fig:block_diagram_summary}
\end{figure}

Recently, there are also many works that aim to improve the tracking performance of quadcopters, particularly for aggressive flights \cite{invernizzi2020comparison,tal2020accurate,foehn2018onboard,FMPC}. 
For example,  \cite{tal2020accurate} tracks high-order derivatives to reject aerodynamic effects present in aggressive flights, \cite{foehn2018onboard} considers a state-dependent linearization of the quadcopter dynamics and applies LQR techniques when the system is operated far from the equilibrium point, and \cite{FMPC} employs flatness results to feedforward-linearize the MPC problem for optimal control. These controllers are usually endowed with stability and/or performance guarantees, but not safety.  Safe controllers based on reinforcement learning and Gaussian processes have been investigated for quadcopters in \cite{lupashin2010simple,singla2019memory,berkenkamp2017safe}. Another line of research that can provide formal safety guarantee in the sense of forward invariance of a given safe set in continuous time is based on control barrier functions (CBFs); see  \cite{ames2016control,cbf_pos,IBVS,rosolia2020unified} for more details.

This paper proposes a convex optimization-based framework that \emph{guarantees continuous-time safety satisfaction} of quadcopters for both trajectory planning and tracking (see Figure \ref{fig:block_diagram_summary}). The contributions of this paper are summarized as follows:
\begin{itemize}
    \item We propose sufficient conditions for the $r$-th derivative of clamped B-spline curves to be included by a convex set. 
    \item We present a second-order cone program (SOCP) framework to solve the safe trajectory planning problem, which includes position, linear velocity, angle, angular velocity, thrust and waypoints constraints, such that the reference trajectory generated satisfies the safety constraints rigorously in the continuous-time sense. 
    Obstacle avoidance is also considered within the SOCP framework.
    \item We propose a CBF-QP-based trajectory tracking method for quadcopters such that the real trajectory is in a prescribed tube of the nominal trajectory in the continuous-time sense. We present sufficient conditions to guarantee feasibility of the CBF-QP. Furthermore, we provide conditions to ensure the CBF-QP-based tracking controller respects thrust, roll angle and pitch angle constraints.
    \item We demonstrate the proposed optimization-based algorithms through multiple experiments using a Crazyflie2.1 nano-quadcopter.
\end{itemize}

The remainder of the paper is organized as follows: Section \ref{Preliminaries} describes the quadcopter dynamics model,  differential flatness, B-spline curves, CBFs and second-order cone constraints; Section \ref{sec:Constraint} provides sufficient conditions for continuous-time set inclusion constraint satisfaction using clamped B-spline curves; Section \ref{sec:traj} formulates a SOCP for safe trajectory planning; Section \ref{sec:track} describes a convex CBF-based QP for safe trajectory tracking; Section \ref{sec:exper} demonstrates the experiments and finally, Section \ref{sec:Conclusions} concludes the paper.

\section{Preliminaries } \label{Preliminaries}
\subsection{Quadcopter Dynamics \& Differential Flatness} \label{Dynamics}
The equations of motion of a quadcopter used in this work are given as follows \cite{min_snap}: 
\begin{subequations} \label{dyn_model}
\begin{align}
    \ddot{\vect{r}} &\!=\! T\vect{z}_B -g\vect{z}_W, \label{dyn_model_acc}\\
    \dot{\vect{\xi}} &\!=\! N^{-1}(\vect{\xi})\vect{\omega},\\
    N(\vect{\xi}) &\!=\! \mat{1 & 0 & -\sin\theta\\
                         0 & \cos\phi & \sin\phi\cos\theta\\
                         0 & -\sin\phi & \cos\phi\cos\theta},\\
    R_B^W &\!=\! \mat{
c\theta c\psi & s\phi s\theta c\psi - c\phi s\psi & s\phi s\psi + c\phi s\theta c\psi\\
c\theta s\psi & s\phi s\theta s\psi + c\phi c\psi & c\phi s\theta s\psi - s\phi c\psi\\
-s\theta & s\phi c\theta & c\phi c\theta}\\
&= \mat{\vect{x}_B & \vect{y}_B & \vect{z}_B},\nonumber
\end{align}
\end{subequations}
where
\begin{align*}
  \vect{z}_W =\mat{0\\0\\1},\quad\vect{r}=\mat{x\\y\\z},\quad  \vect{\xi}=\mat{\phi \\\theta \\\psi},\quad \vect{\omega} =\mat{p \\q \\r},
\end{align*}
are the world-frame's z-axis, the position vector, the Z-Y-X Euler angle vector (see Figure \ref{fig:Coord_Frames}), and the angular velocity vector, respectively, $T$ is the normalized thrust, and $s,c$ stand for $\sin$ and $\cos$ respectively. The considered state and input vectors for a quadcopter are expressed, respectively, as:
\begin{subequations} \label{state_input}
\begin{align}
  \vect{x} &= \mat{x&y&z&\phi&\theta&\psi&\dot{x}&\dot{y}&\dot{z}}^T,\\
  \vect{u} &= \mat{T&p&q&r}^T.
\end{align}
\end{subequations}

\begin{figure}
    \centering
    \includegraphics[width=7cm]{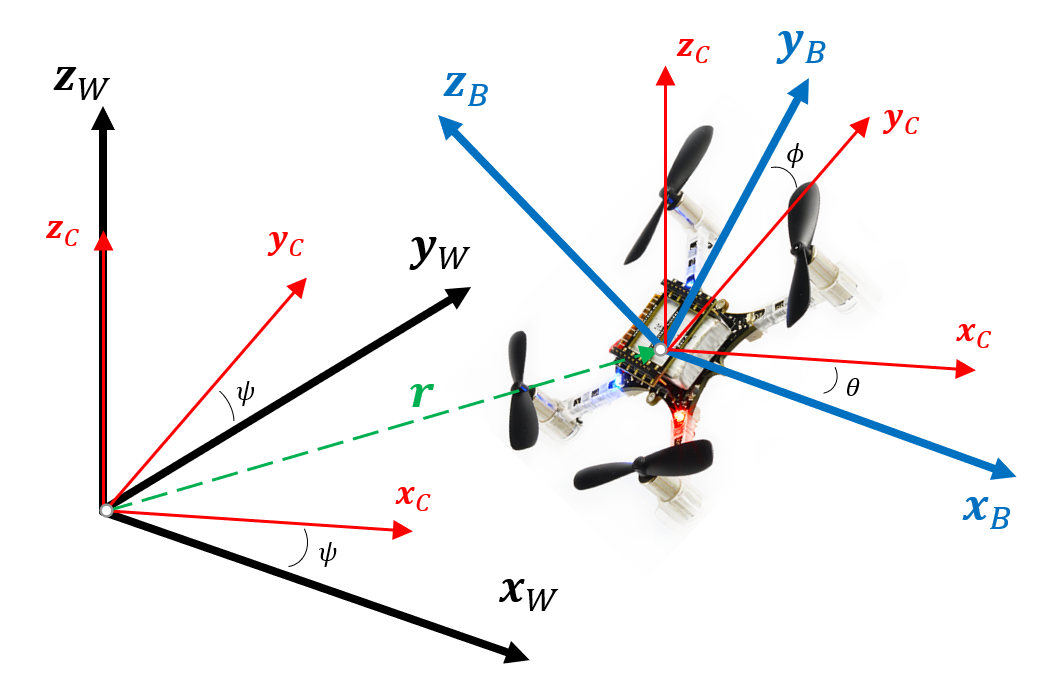}
    \caption{Inertial (or world) coordinate frame $W$ (black) and body coordinate frame $B$ (blue). The intermediate  coordinate frame $C$ (red) is also shown. Crazyflie2.1 figure from \cite{PX4}.}
    \label{fig:Coord_Frames}
\end{figure}

The quadcopter model shown in \eqref{dyn_model} is known to be differentially flat \cite{min_snap,fliess1995flatness}. 
By choosing flat outputs as:
$$
\vect{\sigma} = \mat{x & y & z & \psi}^T,
$$ 
the state and input can be expressed as a function of $\vect{\sigma}$ and its derivatives for some $\Phi$: 
\begin{equation}
    (\vect{x},\vect{u})=\Phi(\vect{\sigma},\dot{\vect{\sigma}},\ddot{\vect{\sigma}},\dddot{\vect{\sigma}}),
    \label{flat_map}
\end{equation}
where the expressions of $\Phi$ are given as \cite{min_snap,zhou2014vector}:
\begin{subequations}
\begin{align}
    \phi &= \sin^{-1}\Big(\frac{\vect{z}_W^T\vect{y}_B}{\cos(\sin^{-1}(\vect{z}_W^T\vect{x}_B))} \Big),\\
    \theta &= -\sin^{-1}(\vect{z}_W^T\vect{x}_B),\\
    T &= \sqrt{\ddot{x}^2 + \ddot{y}^2 + (\ddot{z}+g)^2},\label{flat_T_map}\\
    p &= -\vect{y}_B \cdot \vect{h}_{\omega},\label{flat_p_map}\\
    q &= \vect{x}_B \cdot \vect{h}_{\omega},\label{flat_q_map}\\
    r &= \vect{z}_B \cdot \dot{\psi}\vect{z}_W,
\end{align}
\end{subequations}
with intermediate quantities:
\begin{align*}
    \vect{T} &= \mat{\ddot{x} & \ddot{y} & \ddot{z} + g }^T,\\
    \vect{y}_C &= \mat{-\sin\psi & \cos\psi & 0 }^T,\\
    \vect{z}_B &= \frac{\vect{T}}{\|\vect{T}\|_2},\\
    \vect{x}_B &= \frac{\vect{y}_C \times \vect{z}_B}{\|\vect{y}_C \times \vect{z}_B\|_2},\\
    \vect{y}_B &= \vect{z}_B\times\vect{x}_B,\\
    \vect{h}_{\omega} &= \frac{1}{\|\vect{T}\|_2}(\dddot{\vect{r}}-(\vect{z}_B \cdot \dddot{\vect{r}})\vect{z}_B).
\end{align*}

Generating a trajectory for differentially flat systems reduces to finding a sufficiently smooth flat output trajectory \cite{murray1998real}. In the case of system \eqref{dyn_model}, the flat trajectory $\vect{\sigma}(t)$ needs to be at least thrice-differentiable.


\subsection{B-Spline Curves} \label{B-Splines}
B-splines are widely used to parameterize trajectories because of their nice properties and ease of use. For example, the smoothness of B-splines makes the differentiability constraint of the flat trajectory $\vect{\sigma}(t)$ easy to satisfy, the convexity property allows one to explore state and input constraints in continuous time, and the locality property is useful for re-planning.

Given a \emph{knot vector} $\vect{\tau} = (\tau_0, \ldots,\tau_v)^T$ satisfying $\tau_i \leq \tau_{i+1}$ where $i=0,\ldots,v-1$, a $d$-th degree \emph{B-spline basis} with  $d\in\Z_{> 0}$ is computed recursively as \cite{deboor}:
\begin{align*}
    \lambda_{i,0}(t) &= \begin{cases} 
      1, & \tau_i \leq t <\tau_{i+1}, \\
      0, & \text{otherwise}, \\
   \end{cases}\\
   \lambda_{i,d}(t)&=\frac{t-\tau_i}{\tau_{i+d}-\tau_i}\lambda_{i,d-1}(t)+ \frac{\tau_{i+d+1}-t}{\tau_{i+d+1}-\tau_{i+1}}\lambda_{i+1,d-1}(t).
\end{align*}
In this work, we consider the \emph{clamped, uniform B-spline basis}, which is defined over knot vectors satisfying:
\begin{subequations} \label{clamped_uniform}
\begin{align}
    \text{(clamped)} \quad &\tau_0=\ldots=\tau_d, \quad \tau_{v-d}=\ldots=\tau_{v},\label{clamped}\\
    \text{(uniform)}\quad &\tau_{d+1}-\tau_{d}=\ldots =\tau_{N+1}-\tau_{N},\label{uniform}
\end{align}
\end{subequations}
where $N = v-d-1$. A $d$-th degree \emph{B-spline curve} $\vect{s}(t)$ is a $m$-dimensional parametric curve  built by linearly combining \emph{control points} $\vect{P}_i\in\mathbb{R}^m (i=0,\ldots,N)$ and B-spline basis functions of the same degree. Noting $\vect{s}^{(0)}(t) = \vect{s}(t)$, we generate a B-spline curve and its $r$-th order derivative by:
\begin{equation}
    \vect{s}^{(r)}(t) = \sum_{i=0}^{N}\vect{P}_i\vect{b}_{r,i+1}^T\vect{\Lambda}_{d-r}(t) = PB_r\vect{\Lambda}_{d-r}(t),
    \label{b-spline_curve}
\end{equation}
where the control points are grouped into a matrix
\begin{equation} \label{defP}
    P = \mat{\vect{P}_0&\dots& \vect{P}_N} \in\R^{m\times (N+1)},
\end{equation}
the basis functions are grouped into a vector
\begin{equation}
    \vect{\Lambda}_{d-r}(t) = \left[\lambda_{0,d-r}(t),\dots ,\lambda_{N+r,d-r}(t)\right]^T\in\R^{N+r+1},
\end{equation}
and $\vect{b}_{r,j}^T$ is the $j$-th row of a time-invariant matrix $B_{r}$ whose construction is given in Appendix \ref{B_mat}. 

\begin{definition}\label{virtual_control_points_old} The columns of $P^{(r)} \triangleq PB_{r}$ are called the \emph{$r$-th order virtual control points} (VCPs) of $\vect{s}^{(r)}(t)$ and denoted as $\vect{P}_i^{(r)}$ where $i=0,1,\ldots,N+r$, i.e., 
\begin{equation}\label{virtual_cp_eq}
    P^{(r)} = PB_{r} = \mat{\vect{P}_0^{(r)} & \dots & \vect{P}_{N+r}^{(r)}}.
\end{equation}
\end{definition}
Note that $P^{(0)} = P$ and that the first and last $r$ columns of $P^{(r)}$ are zero vectors because of the form of $B_{r}$. 

Some properties of B-splines that will be used in the following sections are listed below (see \cite{deboor,fajar_thesis} for more details). 
\begin{itemize}
    \item[P1)] \textit{Continuity}. Given a clamped knot vector satisfying \eqref{clamped},  $\vect{s}(t)$ is $C^{\infty}$ for any  $t\notin \vect{\tau}$ and $C^{d-1}$ for any $t\in\vect{\tau}$.
    \item[P2)] \textit{Convexity}. The B-Spline basis functions have the ``partition of unity'' property, meaning that: 
    \begin{itemize}
        \item[(i)] $\lambda_{i,d}(t) \geq 0$;
        \item[(ii)] $\sum_{i=0}^{N}\lambda_{i,d}(t)= 1, \ \forall t\in[\tau_0,\tau_v)$.
    \end{itemize}
    \item[P3)] \textit{Local Support}. Any B-spline basis is only nonzero over a local set of knots: $\lambda_{i,d}(t)\neq 0$ iff $t\in[\tau_i,\tau_{i+d+1})$.
    \item[P4)] \textit{Strong Convexity}. Segments of the B-spline curve are contained within the convex hull of a limited set of control points. Combining each segment we obtain 
    \begin{equation*}
        \vect{s}(t) \in \bigcup_{i=d}^{N}\text{Conv}\{\vect{P}_{i-d},\dots,\vect{P}_i\}.
    \end{equation*}
\end{itemize} 


\subsection{Control Barrier Functions}

The concept of (zeroing)  control barrier functions (CBFs) was first proposed in \cite{Xu2015ADHS}, and was proven to be a powerful control design method to ensure the safety (in the sense of controlled invariance) of control affine systems \cite{ames2016control,xu2018correctness}. A provably safe control law is generated by solving a CBF-QP online \cite{Xu2015ADHS}. In this work, we will use time-varying CBFs with relative degree 2 and their control sharing property \cite{xu2018constrained}. 

Consider a time-varying affine control system
\begin{align}\label{eqn:controlsys}
\dot{\vect{x}} = f(t,\vect{x}) + g(t,\vect{x})u,
\end{align}
where $\vect{x} \in \R^n$, $u \in U\subset \R$ and $f:\R\times \R^n \to \R^n$, $g:\R\times \R^n \to \R^n$ are piecewise continuous in $t$ and locally Lipschitz in $\vect{x}$. 
Given a sufficiently smooth time-varying function $h(t,\vect{x}):\R\times \R^n \to \R$, we define the modified Lie derivative of $h(t,\vect{x})$ along $f$ as $\bar L_f^ih:=(\frac{\pa }{\pa t}+L_f)^ih$  where $i$ is a non-negative integer. 

\begin{definition}\cite{xu2018constrained}\label{dfn:cbfhigh}
Given a control system \eqref{eqn:controlsys}, a $C^r$ function $h(t,\vect{x}): \R\times \R^n \to \R$ with a relative degree $r$ is called a CBF if  there exists a column vector ${\bf a}=(a_1,\ldots,a_r)^T\in\R^r$ such that for all $\vect{x}\in\R^n$, $t\geq 0$,
\begin{align}\label{ineq:ZCBF2}
	& \sup_{u \in U}  [L_g\bar L_f^{r-1}h(t,\vect{x})u +\bar L_f^rh(t,\vect{x})+{\vect{a}}^T\eta(t,\vect{x}) ] \geq 0,
\end{align}
where $\eta(t,\vect{x})=(\bar L_f^{r-1}h,\bar L_f^{r-2}h,\ldots,h)^T\in\R^r$, and the roots of $p_0^r(\la)=\la^r+a_1\la^{r-1}+...+a_{r-1}\la+a_r$ are negative reals $-\la_1,...,-\la_r<0$.
\end{definition}

Define a series of functions:
\begin{align}\label{liftedh}
\nu_0(t,\vect{x})&=h(t,\vect{x}),\;\nu_{k}=\Big(\frac{\diff}{\diff t}+\la_k\Big)\nu_{k-1},\;1\le k\le r.
\end{align}
For any $t\geq 0,\vect{x}\in \R^n$, define the set of admissible inputs satisfying the CBF condition as $\mathcal{K}_h(t,\vect{x}):=\{u\in U|L_g\bar L_f^{r-1}h(t,\vect{x})u +\bar L_f^rh(t,\vect{x})+\vect{a}^T\eta(t,\vect{x}) \ge 0\}$. 

\begin{lemma}\cite{xu2018constrained}\label{thm:zcbf}
Given a control system \eqref{eqn:controlsys}, if 1) $h(t,\vect{x})$ with a relative degree $r$ is a CBF such that \eqref{ineq:ZCBF2} holds and the roots of $p_0^r(\la)=\la^r+a_1\la^{r-1}+...+a_{r-1}\la+a_r$ are  $-\la_1,...,-\la_r<0$, 2) $\nu_i$ defined in \eqref{liftedh} satisfy $\nu_i(0,\vect{x}_0)\ge 0$ for $i=0,1,...,r-1$, then any controller $u(t,\vect{x}) \in \mathcal{K}_h(t,\vect{x})$ that is Lipschitz in $\vect{x}$ will render $h(t,\vect{x})\ge 0$ for any $t\geq 0$.
\end{lemma}

Given a control system \eqref{eqn:controlsys} and $q (q\geq 2)$ CBFs $h_i(t,\vect{x}),i=1,...,q$, the \emph{control-sharing property} of CBFs was proposed in \cite{xu2018constrained} to ensure the feasibility of the CBF-QP that includes multiple CBF constraints. We refer the reader to Definition 2 in \cite{xu2018constrained} for more detail.


\subsection{Second-order Cone Constraints} \label{SOC_OPT} 
A \emph{second-order cone program} (SOCP) is a type of convex optimization problem of the following form \cite{lobo1998applications,alizadeh2003second,boyd2004convex}: 
\begin{align} \label{SOCP_form}
    \min \quad & \vect{f}^T\vect{x}\\
    \text{s.t.}\quad &\|A_i\vect{x}+\vect{b}_i\|_2\leq \vect{c}_i^T\vect{x}+d_i, \quad i=1,\ldots,m,\nonumber
\end{align}
where $\vect{x}\in\mathbb{R}^n$ is the optimization variable, and the problem parameters are $\vect{f}\in\R^n,A_i\in\mathbb{R}^{(n_i-1)\times n},\vect{b}_i\in\R^{n_i-1},\vect{c}_i\in\R^{n}$ 
and $d_i\in\R$. Constraints of the form:
\begin{equation} \label{SOC}
    \|A\vect{x}+\vect{b}\|_2\leq \vect{c}^T\vect{x}+d
\end{equation}
are called \emph{second-order cone} (SOC) constraints  of dimension $n$. Many commonly-seen convex constraints can be formulated as SOC; for example, ellipsoidal constraints ($\vect{x}^TQ\vect{x} \leq r,\ Q\succeq0$) and polyhedral constraints ($A\vect{x}\leq \vect{b}$) \cite{lobo1998applications,alizadeh2003second,boyd2004convex}. SOCP encompasses many types of important convex programs as special cases, such as linear programs, convex QPs and convex quadratically constrained QPs. SOCP can be solved in polynomial time via the interior-point
method and off-the-shelf solvers such as MOSEK exist \cite{potra2000interior,mosek}.


\section{Continuous-Time Set Inclusion of B-spline Curves \& Their Derivatives}\label{sec:Constraint}
In this section, continuous-time convex set inclusion of B-spline curves and their derivatives will be considered. These results provide an intuitive description of the strong convexity property of B-spline curves and their derivatives, which is the key to enforcing the continuous-time state and input constraint satisfaction of quadcopter trajectories.


\begin{lemma} \label{smooth_strong_convexity}
The $r$-th derivative of a clamped B-spline curve defined as in \eqref{b-spline_curve} is contained within the convex hull of its $r$-th order virtual control points such that:
\begin{equation}
    \vect{s}^{(r)}(t) \in \text{Conv}\{\vect{P}_{0}^{(r)},\dots,\vect{P}_{N+r}^{(r)}\}, \quad t\in [\tau_0,\tau_v).
    \label{convex}
\end{equation} 
Furthermore, for $d\leq i\leq N$, 
\begin{equation}        
    \vect{s}^{(r)}(t) \in \text{Conv}\{\vect{P}_{i-d+r}^{(r)},\dots,\vect{P}_{i}^{(r)}\}, \quad  t\in[\tau_i,\tau_{i+1}).
    \label{strong_convex}
\end{equation}
By \eqref{strong_convex}, condition \eqref{convex} can be reduced to:
\begin{equation}
    \vect{s}^{(r)}(t) \in \text{Conv} \{\vect{P}_{r}^{(r)},\dots,\vect{P}_N^{(r)}\}, \quad t\in[\tau_0,\tau_v).
    \label{smooth_strong_convex}
\end{equation}
\end{lemma}
\begin{proof}
Applying 
the definition of VCP to \eqref{b-spline_curve} we obtain:
\begin{equation}
    \vect{s}^{(r)}(t) = P^{(r)}\vect{\Lambda}_{d-r}(t) = \sum_{i=0}^{N+r} \lambda_{i,d-r}(t)\vect{P}^{(r)}_i.
    \label{lin_comb}
\end{equation}
Applying P2) of B-splines shown in Section \ref{B-Splines}, we determine that $\vect{s}^{(r)}(t)$ is a convex combination of $r$-th order VCPs, which is equivalent to \eqref{convex}. Using P3) of B-splines we have 
\begin{equation*}
    \sum_{j=i-d+r}^{i}\lambda_{j,d-r}(t)=1,\quad  t\in[\tau_i,\tau_{i+1}), \quad d\leq i \leq N.
\end{equation*}
Leveraging this result, \eqref{lin_comb} reduces to:
\begin{equation*}
    \vect{s}^{(r)}(t) = \sum_{j=i-d+r}^{i} \lambda_{j,d-r}(t)\vect{P}^{(r)}_j, \forall t\in [\tau_i,\tau_{i+1}), d\leq i\leq N,
\end{equation*}
which is equivalent to \eqref{strong_convex}. Taking the union of all nonempty knot segments in the knot vector $\vect{\tau}$ results in \eqref{smooth_strong_convex}. 
\end{proof}

\begin{proposition} \label{inclusion_prop}
Given a convex set $\mathcal{S}$ and the $r$-th derivative of a clamped B-spline curve $\vect{s}^{(r)}(t)$ defined as in \eqref{b-spline_curve}, if
\begin{equation}
     \vect{P}_{j}^{(r)} \in \mathcal{S}, \quad j=r,\ldots,N
    \label{inclusion_global}
\end{equation}
holds, then $\vect{s}^{(r)}(t) \in \mathcal{S}, \ t\in[\tau_0,\tau_v)$. If
\begin{equation}
    \vect{P}^{(r)}_j \in S, \quad j=i-d+r,\dots,i, \quad i\in\{d,\ldots,N\}
    \label{inclusion_local}
\end{equation}
holds, then $\vect{s}^{(r)}(t)\in \mathcal{S}, \ t\in[\tau_i,\tau_{i+1})$.
\end{proposition}

\begin{proof}
     From \eqref{smooth_strong_convex} and the convexity property, it follows that containing $\vect{s}^{(r)}(t)$ in a convex set $\mathcal{S}$ for $t\in[\tau_0,\tau_v)$ can be achieved by containing $\vect{P}_j^{(r)}$ in $\mathcal{S}$ for $j=r,\ldots,N$. This results in \eqref{inclusion_global}. Furthermore, containing $\vect{s}^{(r)}(t)$ in $\mathcal{S}$ for a nonempty time interval $[\tau_i,\tau_{i+1})$ reduces to containing $d-r+1$ $r$th order VCPs in $\mathcal{S}$ (i.e., those that form the convex hull in \eqref{strong_convex}). This is equivalent with \eqref{inclusion_local}.
\end{proof}

\begin{remark}
Proposition \ref{inclusion_prop} above generalizes Proposition 1 in \cite{Constrained_trajectory} from an interval set inclusion to a more general convex set inclusion. This will be key when considering spherical and conic sets to ensure the safety of quadcopters as will be shown in the following sections.
\end{remark}

\section{Trajectory Planning with Continuous-Time Safety Guarantee} \label{sec:traj}
This section presents a SOCP framework to solve a safe trajectory planning problem that includes various state and input safety constraints, such that the reference trajectory  satisfies the safety constraints rigorously in the continuous-time sense. 
Since the flat output $\psi$ is irrelevant to the state and input constraints considered in this paper, we set it to zero and limit ourselves to finding a sufficiently smooth, position flat output trajectory defined as
$$
\vect{r}(t) = \mat{x(t) & y(t) & z(t)}^T.
$$ 

\subsection{Encoding Safety Specifications as SOC Constraints}
Consider the following state and input safety constraints that have to be satisfied rigorously in continuous-time:
\begin{subequations} \label{state_input_const}
\begin{align}
  \vect{x}(t) &\in\mathcal{X} = \mathcal{X}_{p} \cap \mathcal{X}_{v} \cap\mathcal{X}_{\xi},\\
  \vect{u}(t) &\in \mathcal{U} = \mathcal{U}_T\cap\mathcal{U}_{\omega},
\end{align}
\end{subequations}
where the state $\vect{x}(t)$ and input $\vect{u}(t)$ are defined in \eqref{state_input}, and $\mathcal{X}_{p},\mathcal{X}_{v},\mathcal{X}_{\xi},\mathcal{U}_T,\mathcal{U}_{\omega}$ are the constraint sets for the position, linear velocity, roll and pitch angles, total thrust, and angular velocity, respectively. For notation simplicity we first consider \eqref{state_input_const} over the entire time interval $t\in[\tau_0,\tau_v)$. The procedure will be to formulate each constraint in flat space and then use Proposition \ref{inclusion_prop} to derive finitely many sufficient conditions for \eqref{state_input_const}. In Section \ref{subavoidance}, we will discuss how to use Proposition \ref{inclusion_prop} to achieve obstacle avoidance by satisfying the constraints shown in \eqref{state_input_const} on sub-intervals instead. Apart from the continuous-time safety constraints, we will also consider waypoint constraints that have to be satisfied at certain discrete time instances.
\subsubsection{Position Constraints $\vect{x}\in\mathcal{X}_p$}
The position constraint set $\mathcal{X}_p$ encodes limited flight volumes as follows:
\begin{equation}\label{const_pos}
    \mathcal{X}_{p} = \{\vect{x}\in\mathbb{R}^9\mid\vect{r} \in \mathcal{K}_p\},
\end{equation}
where $\mathcal{K}_p$ is a given SOC.
Proposition \ref{inclusion_prop} immediately provides sufficient conditions for $\vect{x}(t)\in\mathcal{X}_p,\ t\in[\tau_0,\tau_v)$:
\begin{equation}\label{cp_pos}
    \vect{P}_{j} \in \mathcal{K}_p, \quad j=0,\dots,N.
\end{equation}
Note that we restrict ourselves to SOC sets to preserve the overall SOCP structure that will arise from the rest of the constraints. 

\subsubsection{Linear Velocity Constraints $\vect{x}\in\mathcal{X}_v$} Limiting flight velocities can  reduce the risk of injury or equipment damage. Consider the linear velocity constraint set defined as:
\begin{equation}\label{const_vel}
    \mathcal{X}_{v} = \{\vect{x}\in\mathbb{R}^{9}:\|\dot{\vect{r}}\|_2\leq \overline{v}\},
\end{equation} 
where $\overline{v}\geq0$ denotes the allowable bound of the maximal speed. Using Proposition 1, the constraint $\vect{x}(t)\in\mathcal{X}_v, \ t\in[\tau_0,\tau_v)$ can be formulated as a constraint on the $1$-st order VCPs: 
$$
\|\vect{P}_j^{(1)}\|_2 \leq \overline{v},\ j=1,\ldots,N,
$$ 
which can be expressed equivalently as a SOC constraint on the control points:
\begin{equation} \label{cp_vel}
    \bigg\|\sum_{i=0}^N\vect{P}_ib_{1,i+1,j+1}\bigg\|_2\leq \overline{v}, \quad j=1,\ldots,N,
\end{equation}
where $b_{r,i,j}$ is the $(i,j)$-th entry of matrix $B_r$.



\subsubsection{Angular Constraints $\vect{x}\in\mathcal{X}_{\xi}$}  
Providing bounds for the roll and pitch angles can enhance system stability and improve the accuracy of the dynamics model \eqref{dyn_model} by limiting aerodynamic effects. This is especially relevant when implementing controllers that leverage linearizations (e.g., LQR or linear MPC) or small-angle approximations. Consider the angular constraints set defined as follows:
\begin{equation} \label{const_ang}
    \mathcal{X}_{\xi}=\{\vect{x}\in\mathbb{R}^9:|\phi|,|\theta| \leq \epsilon\},
\end{equation}
where $\epsilon< \pi/2$ is a given bound for the roll and pitch angles. We begin by exposing the geometric meaning of a result from literature on quadcopter differential flatness. 
\begin{lemma}\label{epsilon_cone_lemma} 
Given a positive scalar $0<\epsilon<\pi/2$, and the following SOC:
\begin{align}\label{ang_bounds_cone}
    \mathcal{K}_{\epsilon} &\triangleq \{\vect{p}\in\mathbb{R}^3:\|A_{\epsilon}\vect{p}\|_2\leq p_3+g\},\\
    A_{\epsilon} &=|\cot\epsilon|\mat{\vect{x}_W & \vect{y}_W & \vect{0}_{3\times 1}},\nonumber
\end{align}
if the acceleration signal $\ddot{\vect{r}}(t)\in\mathcal{K}_{\epsilon}$,
then $|\phi(t)|,|\theta(t)|\leq \epsilon$ regardless of the yaw angle $\psi(t)$.
\end{lemma}
\begin{proof} 
Let $\ddot{\vect{r}}(t)\in\mathcal{K}_{\epsilon}$ and expand the SOC \eqref{ang_bounds_cone} to arrive at:
\begin{equation}
    \ddot{x}^2(t)\cot^{2}\epsilon + \ddot{y}^2(t)\cot^2\epsilon \leq (\ddot{z}(t)+g)^2.
\end{equation}
Let $k_1(t)\triangleq \ddot{x}(t)$, $k_2(t)\triangleq \ddot{y}(t)$ and $k_3(t) \triangleq \ddot{z}(t)+g$. Dividing the inequality above by $\cot^2{\epsilon}\neq 0$ yields:
\begin{equation*}
    \frac{k_1^2(t)+k_2^2(t)}{k_3^2(t)}\leq \tan^2\epsilon.
\end{equation*}
Then, the conclusion follows from  Proposition 2 in \cite{effective_ang}.
\end{proof}
Note that \eqref{ang_bounds_cone} describes the interior of a conic surface in $\ddot{\vect{r}}$ coordinates with apex at $(0,0,-g)$. We have the following result for angular constraints satisfaction.
\begin{lemma}
Consider B-spline curve $\vect{r}(t)$ defined as in \eqref{b-spline_curve}. If
\begin{align} \label{cp_ang}
    \bigg\|A_{\epsilon}\sum_{i=0}^N\vect{P}_ib_{2,i+1,j+1}\bigg\|_2 \leq g + \sum_{i=0}^N\vect{z}_W^T\vect{P}_ib_{2,i+1,j+1}
\end{align}
holds for $j=2,\ldots,N$ with $A_{\epsilon}$ given as in \eqref{ang_bounds_cone}, then $\vect{x}(t)\in\mathcal{X}_{\xi}, \forall  t\in[\tau_0,\tau_v)$.
\end{lemma}

\begin{proof} Applying the definition of $2$nd order VCPs given in \eqref{virtual_cp_eq} to condition \eqref{cp_ang} yields:
\begin{equation} \label{vcp_ang}
\|A_{\epsilon}\vect{P}_j^{(2)}\|_2 \leq \vect{z}_W^T\vect{P}_j^{(2)}+g, \quad j=2,\dots,N.
\end{equation}
By Proposition \ref{inclusion_prop}, we have:
\begin{equation}\label{lem3ineq2}
    \|A_{\epsilon}\vect{r}^{(2)}(t)\|_2\leq \vect{z}_W^T\vect{r}^{(2)}(t) + g, \quad t\in[\tau_0,\tau_v).
\end{equation}
Since condition \eqref{lem3ineq2} is equivalent to $\ddot{\vect{r}}(t)\in\mathcal{K}_{\epsilon}$ with $\mathcal{K}_{\epsilon}$ given as in \eqref{ang_bounds_cone}, it follows from Lemma \ref{epsilon_cone_lemma} that $|\phi(t)|,|\theta(t)|\leq \epsilon$, $\forall t\in[\tau_0,\tau_v)$, which means that $\vect{x}(t)\in\mathcal{X}_{\xi}, \forall  t\in[\tau_0,\tau_v)$.
\end{proof} 


\subsubsection{Thrust Constraints $\vect{u}\in\mathcal{U}_{T}$} 
Upper-bounding the total thrust is needed to prevent actuator saturation, while lower-bounding the total thrust is common in aerospace systems, e.g., to facilitate  soft-landing  \cite{accikmecse2011lossless}. Consider the thrust constraints  set as follows:
\begin{equation}\label{const_T}
    \mathcal{U}_T = \{\vect{u}\in\mathbb{R}^4\mid\underline{T}\leq T \leq \overline{T}\},
\end{equation}
where $\underline{T},\overline{T}$ are given constants satisfying $0\leq \underline{T}\leq g\leq \overline{T}$.
\begin{lemma}\label{lemthrust} 
Consider B-spline curve $\vect{r}(t)$ defined as in \eqref{b-spline_curve}. If
\begin{subequations}\label{cp_thrust}
\begin{align}
    \bigg\|\sum_{i=0}^N\vect{P}_ib_{2,i+1,j+1} + g\vect{z}_W\bigg\|_2\leq \overline{T}, \quad j=2,\ldots,N, \label{cp_Tmax}\\
    \sum_{i=0}^N\vect{z}_W^T\vect{P}_ib_{2,i+1,j+1}\geq \underline{T}-g, \quad j= 2,\ldots,N, \label{cp_Tmin}
\end{align}
\end{subequations}
hold, then $\vect{u}(t) \in \mathcal{U}_T, \; \forall t\in[\tau_0,\tau_v)$.
\end{lemma}


\begin{proof}
Applying the definition of $2$nd order VCPs \eqref{virtual_cp_eq} to \eqref{cp_Tmax} yields:
\begin{equation} \label{vcp_Tmax}
    \|\vect{P}_j^{(2)}+g\vect{z}_W\|_2\leq \overline{T}, \quad j = 2,\dots,N.
\end{equation}
By Proposition \ref{inclusion_prop}, and expanding the norm shown in \eqref{vcp_Tmax}, we have:
\begin{equation} \label{T_sphere}
    \sqrt{\ddot{x}^2(t) + \ddot{y}^2(t)+\big(\ddot{z}(t)+g\big)^2}\leq \overline{T}, \quad t\in[\tau_0,\tau_v),
\end{equation}
which constrains the acceleration to lie within a sphere in $\ddot{\vect{r}}$ coordinates centered at $(0,0,-g)$. Note that by the flatness map $\eqref{flat_T_map}$, \eqref{T_sphere} is equivalent to $T(t)\leq \overline{T},\ t\in[\tau_0,\tau_v)$. We now apply the same initial steps to \eqref{cp_Tmin} to get
\begin{equation}\label{vcp_Tmin}
    \vect{P}_j^{(2)}\geq \underline{T}-g, \quad j=2,\ldots,N,
\end{equation}
and applying Proposition 1 to get
\begin{equation}
    \ddot{z}(t)\geq \underline{T}-g, \quad t\in[\tau_0,\tau_v),
\end{equation}
which lower-bounds the $z$-component of acceleration. Therefore,
\begin{equation}
    \underline{T}\leq \sqrt{\ddot{x}^2(t) + \ddot{y}^2(t)+\big(\ddot{z}(t)+g\big)^2}=T(t)
\end{equation}
holds over $t\in[\tau_0,\tau_v)$. This completes the proof.
\end{proof}


\begin{remark}
When the inequality $\underline{T}\leq T(t)$ is expressed in terms of the acceleration $\ddot{\vect{r}}(t)$ by \eqref{flat_T_map}, it signifies exclusion from a spherical region. As a result, it is a nonconvex constraint in terms of the B-spline control points. 
Such a constraint was converted to a mixed-integer constraint in \cite{Constrained_trajectory}, and its conservative relaxation  was also considered in \cite{mueller2013model}. The spherical region \eqref{T_sphere} is bigger than typical quadcopter operation requires: the lower half ($\ddot{z} < -g$) of sphere \eqref{T_sphere} is accessed only when the quadcopter is ``upside-down'' ($|\phi|\geq \pi/2$ or $|\theta| \geq \pi/2$) and this is only required in very aggressive flights. The proof of Lemma \ref{lemthrust} reveals that some  generally unused feasible region is  sacrificed by lower-bounding $\ddot{z}(t)$ which results in a lower-bound for $T(t)$. 
\end{remark}

\subsubsection{Angular Velocity Constraints $\vect{u}\in\mathcal{U}_{\omega}$} 
Angular velocities are the inputs to the quadcopter model considered and are usually subject to saturation. Limits in angular velocities are also desirable to ensure integrity of gyroscope data. Consider the angular velocity constraints set defined as follows:
\begin{equation}\label{const_dang}
    \mathcal{U}_{\omega} = \{\vect{u}\in\mathbb{R}^4:|p|,|q| \leq \overline{\omega}\},
\end{equation}
where $\overline{\omega}$ is a given bound for the angular velocities. Recall that $r(t)=0$ because of the assumption $\psi(t) = 0$. 
\begin{lemma}\label{lemangular}
Consider B-spline curve $\vect{r}(t)$ defined as in \eqref{b-spline_curve}. For any vector  $\vect{\zeta}=(\zeta_1,...,\zeta_{N-d+1})^T$ whose entries are positive constants, if the following conditions:
\begin{subequations} \label{cp_omega}
\begin{align}
    \sum_{i=0}^N\vect{z}_W^T\vect{P}_ib_{2,i+1,j+1} &\geq \zeta_{l-d+1}-g,\quad j = l-d+2,\ldots,l, \label{cp_omega_Tmin}\\
    \bigg\|\sum_{i=0}^N\vect{P}_ib_{3,i+1,j+1}\bigg\|_2 &\leq \overline{\omega}\ \zeta_{l-d+1},\quad j=l-d+3,\ldots,l,\label{cp_omega_jerk}
\end{align}
\end{subequations}
hold for $l=d,\ldots,N$, then $\vect{u}(t)\in\mathcal{U}_{\omega},\; \forall t\in[\tau_0,\tau_v)$.
\end{lemma}
\begin{proof}
Applying the definitions of $2$nd order VCPs to \eqref{cp_omega_Tmin} and $3$rd order VCPs to \eqref{cp_omega_jerk} yields:
\begin{subequations}
\begin{align}
    \vect{z}_W^T\vect{P}_{j}^{(2)}\geq \zeta_{l-d+1}-g,  \quad j = l-d+2,\dots,l,\label{vcp_omega_Tmin}\\
    \|\vect{P}^{(3)}_j\|_2\leq \zeta_{l-d+1}\overline{\omega},\quad j = l-d+3,\dots,l.\label{vcp_omega_jerk}
\end{align}
\end{subequations}
Applying now \eqref{inclusion_local} from Proposition \ref{inclusion_prop} allows us to establish:
\begin{equation}\label{dang_bounds_lax_local}
T(t) \geq \zeta_{l-d+1},\; \|\dddot{\vect{r}}(t)\|_2 \leq \overline{\omega}\ \zeta_{l-d+1}, \quad t\in[\tau_l,\tau_{l+1}),
\end{equation}
which lower-bounds the thrust and upper-bounds the magnitude of the jerk over the $l$-th knot sub-interval. Recalling that the entries of $\vect{\zeta}$ are positive, we can combine inequalities to write:
\begin{equation*}
    \frac{\|\dddot{\vect{r}}(t)\|_2}{T(t)}\leq\frac{\|\dddot{\vect{r}}(t)\|_2}{\zeta_{l-d+1}}\leq \overline{\omega},\quad t\in[\tau_l,\tau_{l+1}). 
\end{equation*}
Now examine the flat maps \eqref{flat_p_map} and \eqref{flat_q_map} and take the absolute value of both sides to arrive at:
\begin{equation*}
    |p| = |-\vect{y}_B\cdot\vect{h}_{\omega}|, \quad
    |q| = |\vect{x}_B\cdot\vect{h}_{\omega}|.
\end{equation*}
Since $\vect{x}_B$ and $\vect{y}_B$ are unit vectors, we can write $|p|,|q|\leq \|\vect{h}_{\omega}\|_2$ by the definition of the dot product. Furthermore, examining $\vect{h}_{\omega}$ we note that the numerator is a projection of the jerk to a plane perpendicular to the $\ddot{z}$-axis which only reduces its norm, and that the denominator is exactly the thrust $T$. In particular, the following holds:
\begin{equation*}
    |p(t)|,|q(t)|\leq \|\vect{h}_{\omega}(t)\|_2\leq \frac{\|\dddot{\vect{r}}(t)\|_2}{T(t)}\leq \overline{\omega}, \quad  t\in[\tau_l,\tau_{l+1}).
\end{equation*}
Finally, taking $l=d,\ldots,N$ and letting $\vect{\tau}$ be a clamped knot vector satisfying \eqref{clamped} allows to establish $|p(t)|,|q(t)|\leq \overline{\omega}$ over the whole interval $[\tau_0,\tau_v)$. This completes the proof.
\end{proof}
\subsubsection{Waypoint Constraints}
Waypoint constraints are commonly used for shaping the trajectory.  
Assume that there are $n_{wp}$ waypoints $\vect{p}_1^{wp},\ldots,\vect{p}_{n_{wp}}^{wp}\in\mathbb{R}^3$ near which the trajectory must pass at certain discrete times $t_i, (i=1,\ldots,n_{wp})$. Specifically, the following constrains are enforced for the position:
\begin{equation}
    \|\vect{r}(t_i)-\vect{p}^{wp}_i\|_2\leq d_{wp}, \quad i=1,\dots,n_{wp}
    \label{const_wp},
\end{equation}
where $d_{wp}\geq 0$ is the desired radius indicating the closeness of the trajectory to the waypoints. Note that the waypoint constraints are exact when $d_{wp}=0$. Recalling that $\vect{r}(t)$ is considered a B-spline curve, constraint \eqref{const_wp} is a SOC constraint in terms of its control points:
\begin{equation}\label{cp_wp}
    \bigg\|\sum_{j=0}^{N}\vect{P}_j\lambda_{j}(t_i) - \vect{p}_{i}^{wp}\bigg\|_2\leq d_{wp}, \quad i = 1,\dots,n_{wp}.
\end{equation}

\subsection{Safe Trajectory Planning as SOCP}

Suppose that the initial (resp. final) conditions of the trajectory are fixed and given as $\vect{p}^{0}_r$ (resp. $\vect{p}^{f}_r$) where $r = 0,\dots,n_0$ (resp. $r=0,\dots,n_f$). For example, the initial (resp. final) position of the trajectory is denoted as $\vect{p}_0^0$ (resp. $\vect{p}_0^f$) for $r=0$, and the initial (resp. final) velocity is denoted as $\vect{p}^{0}_1$ (resp. $\vect{p}^{f}_1$) for $r=1$. 
These initial/final constraints can be expressed as:
\begin{equation}\label{extremes}
    \sum_{i=0}^N\vect{P}_i\vect{b}_{r,i+1}^T\vect{\Lambda}_{d-r}(t_m) = \vect{p}^m_{r}, \quad r = 0,\dots,n_m,
\end{equation}
where $m\in\{0,f\}$ represents initial and final labels, respectively, and $0\leq n_0,n_f\leq d$. Note that $n_0$ may be different from $n_f$ and that $t_0 = \tau_0$ and $t_f = \tau_v$.

In this work we choose an objective function that minimizes $\int_{\tau_0}^{\tau_v}\|\vect{r}^{(4)}(t)\|^2_2 \mathrm{d}t$ where $\vect{r}^{(4)}(t)$ denotes the ``snap'' of the trajectory \cite{min_snap}. 
Taking $\vect{r}^{(r)}$ to be a B-spline curve constructed as in \eqref{b-spline_curve}, the (simple) objective function is given as 
\begin{equation*}
    J_s(\vect{P}_{0},\dots,\vect{P}_{N}) = \int_{t_0}^{t_f}\bigg\|\sum_{i=0}^N\vect{P}_i\vect{b}_{d,4,i+1}^T\vect{\Lambda}_{d-4}(t)\bigg\|_2^2 \mathrm{d}t,
\end{equation*}
where the decision variables are the control points $\vect{P}_0,\ldots,\vect{P}_N$ of the B-spline curve $\vect{r}(t)$. 
By examining \eqref{dang_bounds_lax_local} we observe that it is desirable to make each entry $\zeta_k$ large to maximize the feasible region. Therefore, we  introduce $\vect{\zeta}$ into the objective function such that it is maximized. Specifically, the modified objective function is chosen to be:
\begin{equation} \label{obj_dang_bounds}
    J(\vect{P}_0,\dots,\vect{P}_{N},\vect{\zeta}) = J_s(\vect{P}_0,\dots,\vect{P}_{N})- \sum_{k=1}^{N-d+1}\zeta_{k}.
\end{equation}




Summarizing results above, the safe trajectory planning problem can  now be expressed as the following  SOCP:
\begin{align} \label{FLAT-SOCP}
\tag{FLAT-SOCP}
    \min_{\vect{P}_{0},\dots,\vect{P}_{N},\vect{\zeta}} \quad &J(\vect{P}_0,\dots,\vect{P}_N,\vect{\zeta}) \\
    \text{s.t.} \quad & \eqref{cp_pos},\eqref{cp_vel},\eqref{cp_ang},\eqref{cp_thrust},\eqref{cp_omega},\eqref{cp_wp},\eqref{extremes}\;\text{hold},\nonumber
\end{align}
where ``FLAT'' means that the solutions parameterize the flat output trajectory $\vect{\sigma}(t) = (\vect{r}(t),0)$. 
The trajectory in flat space is defined by control points generated from \eqref{FLAT-SOCP}, and the corresponding state/input trajectory satisfies all the safety constraints. This is summarized in the following theorem.
\begin{theorem}\label{thm1}
Suppose that \eqref{FLAT-SOCP} has a feasible solution $\vect{P}_0,\dots, \vect{P}_N$. Let $\vect{r}(t)$ be a B-spline curve defined by the control points $\vect{P}_0,\dots, \vect{P}_N$. Then, the quadcopter state $\vect{x}(t)$ and input $\vect{u}(t)$ trajectories, which are obtained from the flat map \eqref{flat_map} with $\vect{\sigma}(t) = (\vect{r}(t),0)$,
satisfy the safety constraints \eqref{state_input_const} for all $t\in[\tau_0,\tau_v)$.
\end{theorem}

Problem \eqref{FLAT-SOCP} can be relaxed to a QP if one replaces the convex regions described by the cone constraints in \eqref{FLAT-SOCP} by their polytopic inner approximations. For example, for the SOC constraint \eqref{cp_Tmax}, if we find a polytopic inner approximation as follows: 
$$
\{\ddot{\vect{r}}\in\mathbb{R}^3:A_T\ddot{\vect{r}}\leq \vect{b}_{T}\}\subset\{\ddot{\vect{r}}\in\mathbb{R}^3:\eqref{T_sphere}\;\mbox{holds}\},
$$
then we can replace the SOC constraint \eqref{cp_Tmax} with the linear constraint 
$$
A_T\vect{P}_j^{(2)}\leq \vect{b}_T,\quad  j=2,\dots,N,
$$ 
in \eqref{FLAT-SOCP}. Other replacements can be done in a similar way for all SOC constraints.

\begin{remark}
The optimization formulation shown in \eqref{FLAT-SOCP} is based on B-spline parameterization of the flat trajectory, which is more convenient than the parameterization using piece-wise polynomials as in  \cite{min_snap}. By property P1), the smoothness requirement from the flatness map is satisfied automatically by choosing $d\geq4$, without requiring additional smoothness constraints as in \cite{min_snap} or objective function reformulation as in \cite{polynomial}. 
\end{remark}

\begin{remark}
The feasible region of \eqref{FLAT-SOCP} may be expanded at the expense of computational burden by increasing $N$ (number of control points), which must always respect $v = N+d+1$.
\end{remark}

\begin{remark}
In \cite{Flat_trajectory_design}, the problem of ensuring continuous-time safety for thrust and Euler angles is considered via solving a nonlinear optimization problem, for which only local minimum can be found.  In contrast, \eqref{FLAT-SOCP} is a convex program whose optimal solution can be found in polynomial time.
\end{remark}

\begin{remark} \label{Local_cst_rmk}
Theorem \ref{thm1} ensures that the trajectory generated respects the safety constraints \eqref{state_input_const} for all $t\in[\tau_0,\tau_v)$. 
However, Proposition \ref{inclusion_prop} provides a framework for satisfying the constraints for $t\in[t_1,t_2]\subseteq[\tau_0,\tau_v)$ as long as we can find two knots $\tau_i, \tau_j$ such that $[t_1,t_2]\subseteq [\tau_i,\tau_j)$. This type of local constraint satisfaction may be of interest in applications where safety is needed only at certain times as well as in obstacle avoidance which will be shown in the next subsection.  
\end{remark}

\begin{remark}
Constraint \eqref{cp_omega_Tmin} is compatible with \eqref{cp_Tmax} and \eqref{cp_Tmin}. When $\underline{T}$ is provided, it follows that $\underline{T} \leq \zeta_k \leq \overline{T}, k = 1,\dots,N-d+1$. In fact, the present formulation may lower-bound $T(t)$ more tightly than required to expand the feasible region for \eqref{dang_bounds_lax_local}, but never so much that it violates the upper bound. 
Additionally, while constraint \eqref{cp_omega} increases optimality by expanding the feasible region, it comes at the expense of computational cost. A possible compromise is by replacing the vector $\vect{\zeta}$ with a scalar $\zeta_\omega$ in the objective function \eqref{obj_dang_bounds} and replacing constraint \eqref{cp_omega} with the following global constraint: 
\begin{align*}
    \sum_{i=0}^N\vect{z}_W^T\vect{P}_ib_{2,i+1,j+1}&\geq \zeta_\omega-g, \quad &j= 2,\dots,N,\\
    \bigg\|\sum_{i=0}^N\vect{P}_ib_{3,i+1,j+1}\bigg\|_2 &\leq \overline{\omega}\ \zeta_{\omega}, \quad &j=3,\dots,N.
\end{align*}

\end{remark}


\subsection{Obstacle Avoidance with Continuous-time Guarantees via SOCP}\label{subavoidance}



The safe flight space is usually non-convex which makes the obstacle avoidance problem difficult \cite{tang2021real}. Obstacle avoidance with continuous-time safety guarantees was studied in \cite{stoical2016obstacle} where each obstacle is expressed as a polytope and the trajectory planning problem is formulated as a mixed-integer QP, which is an NP-hard problem and renders the online trajectory re-planning computationally infeasible. In this subsection, we show that, by leveraging the locality property of B-splines, \eqref{FLAT-SOCP} can be modified to achieve obstacle avoidance with a rigorous continuous-time safety guarantee.

Suppose that the safe flight space $\mathcal{F} \in \mathbb{R}^3$ is a set defined as $\mathcal{F} \triangleq \mathbb{R}^3\setminus\left(\mathcal{O}_1\cup...\cup\mathcal{O}_{n_o}\right)$ where $\mathcal{O}_i\in \mathbb{R}^3$ denotes the overapproximation of the $i$-th obstacle. Although $\mathcal{F}$ is generally nonconvex, mild assumptions about the distribution of the obstacles allow  us to determine the existance of convex subsets that, together, form a ``safe corridor'' in $\mathcal{F}$ for the trajectory (see, for example, \cite{gao2018online}). The union of convex subsets can approximate $\mathcal{F}$ arbitrarily well. In practice, these convex subsets are easily found with sensors such as stereo cameras and lidars \cite{hornung13auro}. The time-optimality of trajectories generated using ``safe corridors'' remains an active research topic  \cite{tang2019time,sun2020fast}. 


\begin{proposition} \label{convex_sets_prop}
Suppose that there exist $n_s$ overlapping convex sets $\{\mathcal{S}_1,\ldots,\mathcal{S}_{n_s}\}$ inside the safe flight space, i.e., $\mathcal{S}_l\subset \mathcal{F},\ l=1,\ldots,n_s,$ satisfying $\mathcal{S}_i\cap \mathcal{S}_{i+1}\neq \emptyset, \ i = 1,\ldots,n_s-1$ and consider B-spline curve $\vect{r}(t)$ defined as in \eqref{b-spline_curve} with $N=n_s+d-1$. If the following conditions hold:
\begin{equation}\label{eqobstc}
    \vect{P}_{j-1} \in \mathcal{S}_{l},\quad j = l,\ldots,l+d, \quad l = 1,\ldots,n_s,
\end{equation}
then  $\vect{r}(t)\in\mathcal{F}$  for all $t\in[\tau_0,\tau_v)$ and $\vect{r}(t)$ is at least $C^{d-1},\  \forall t\in[\tau_0,\tau_v)$.
\end{proposition}
\begin{proof} Applying Proposition \ref{inclusion_prop} to the inclusion constraint of the $l$-th  group of control points results in $\vect{r}(t)\in \mathcal{S}_l, \ t\in[\tau_{l+d-1},\tau_{l+d})$, which means that the $l$-th segment of the B-spline curve $\vect{r}(t)$ is contained within $\mathcal{S}_l$. Taking now $l=1,\ldots,n_s$ and recalling that we consider a clamped knot vector \eqref{clamped} results in $\vect{r}(t)\in\mathcal{F},\ t\in[\tau_0,\tau_v)$. Further, note that $\mathcal{S}_l\cap \mathcal{S}_{l+1} \neq \emptyset$ allows $\vect{r}(t)$ to be continuous despite the segment-wise constraints. 
\end{proof}

%

Convex sets $\mathcal{S}_1,\ldots,\mathcal{S}_{n_s}$ in Proposition \ref{convex_sets_prop} play a significant role in the properties of the resulting B-spline curve (such as the magnitude of its derivatives). Finding these sets is considered the task of a high-level planner, and is out of the scope of this work. The reader is referred to \cite{tang2019time,sun2020fast} for more detail. Note that the constraints shown in  \eqref{eqobstc}  can be readily added to \eqref{FLAT-SOCP} without compromising the problem structure if the convex sets $\{\mathcal{S}_l\}$ are restricted to SOC form, as will be demonstrated with experiments in Section \ref{sec:exper}.


\begin{remark}
In Section \ref{sec:track} we will establish safety guarantees for bounded reference position trajectory tracking in the infinity-norm sense. This bound at the tracking level can be incorporated when finding the sets $\{\mathcal{S}_l\}$ so that obstacle avoidance is guaranteed also during tracking. Alternatively, one can appropriately enlarge the size of obstacles $\{\mathcal{O}_i\}$.
\end{remark}

\section{Trajectory Tracking with Continuous-Time Safety Guarantee} \label{sec:track}
In this section, we propose a trajectory tracking method for quadcopters with position safety guarantees in continuous-time. The idea is to use CBFs as a safety filter to guarantee boundedness between the real trajectory and the nominal safe trajectory $\vect{r}^{\text{ref}}(t)$ generated from \eqref{FLAT-SOCP}. The safe tracking controller is obtained by solving a convex QP online, whose feasibility will be ensured by the control-sharing property of multiple CBFs. 
\subsection{6-Dim Quadcopter Model \& Feedback Linearization}
We assume now that the angular velocity dynamics are regulated by some high-bandwidth controller and consider the reduced state and input vectors:
\begin{subequations}
\begin{align}
    \vect{z}&= \mat{x & y & z & \dot{x} & \dot{y} & \dot{z}}^T,\\
    \vect{v} &= \mat{T & \phi & \theta &  \psi}^T,\label{reduced_input}
\end{align}
\end{subequations}
subject to \eqref{dyn_model_acc}. Then, we can consider a set of virtual inputs $\vect{\mu} =(\mu_1,\mu_2,\mu_3)^T\triangleq (\ddot{x},\ddot{y},\ddot{z})^T$ found as in \eqref{dyn_model_acc}:
\begin{equation} \label{def_Psi}
    \vect{\mu} = \Psi(\vect{v}) \triangleq T\vect{z}_B-g\vect{z}_W.
\end{equation}
Note that the map is invertible if we recover the yaw angle such that  $\vect{v} = \Psi^{-1}(\vect{\mu},\psi)$ where  $\Psi^{-1}$ is given as: 
\begin{subequations}
\begin{align}
    T &= \sqrt{\mu_1^2 + \mu_2^2 + (\mu_3+g)^2} = \frac{\mu_3+g}{\cos\phi \cos\theta},\label{inverpsi1}\\
    \phi &= \arctan\Big(\frac{\mu_1\sin\psi-\mu_2\cos\psi}{(\mu_3+g)\cos\theta}\Big),\label{inverpsi2}\\
    \theta&= \arctan\Big(\frac{\mu_1 \cos\psi + \mu_2 \sin\psi}{\mu_3+g}\Big).\label{inverpsi3}
\end{align}
\end{subequations}
With the virtual inputs, the quadcopter dynamics becomes a double-integrator \cite{FMPC,cbf_pos}: 
\begin{equation} \label{lin_sys}
    \dot{\vect{z}}(t) = f(\vect{z}(t)) + g(\vect{z}(t)) \vect{\mu}(t) =  A\vect{z}(t) + B\vect{\mu}(t)
\end{equation}
where 
\begin{equation*}
    A = \mat{\vect{0}_3 & I_3\\ \vect{0}_3 & \vect{0}_3}, \quad B=\mat{\vect{0}_3\\ I_3}.
\end{equation*}
Assume that a reference state $\vect{x}^{\text{ref}}(t)$ and a reference input $\vect{u}^{\text{ref}}(t)$ are generated as in Sec. \ref{sec:traj} and note that we can immediately extract $\vect{z}^{\text{ref}}(t)$ and $\vect{v}^{\text{ref}}(t)$ from them. The dynamic feasibility of the reference trajectory will be necessary to establish the forthcoming results. Assume also a \emph{nominal controller} $\vect{v} = \pi(\vect{z})$ is given, which is potentially unsafe. In the following subsection, a CBF-QP-based tracking controller will be designed such that the real trajectory lies within a prescribed tube of the nominal trajectory while the tracking control is close to the nominal controller as much as possible.  

\subsection{Safe Trajectory Tracking via CBF-QP Controller}
Consider the following time-varying safe set:
\begin{equation}\label{eqS}
    \mathcal{S}_h(t) = \{\vect{z}\in\mathbb{R}^6: \|\vect{r}-\vect{r}^{\text{ref}}(t) \|_{\infty} \leq \delta\}, \quad t\in[\tau_0, \tau_v),
\end{equation}
where $\delta>0$ represents the maximum allowable deviation in each axis. That is, the trajectory is considered as safe if it is within a \emph{safe tube} of radius $\delta$ around the nominal trajectory. Note that it is straightforward to generalize the following results to having different maximum bounds for each axis. 

To ensure $\vect{z}(t)\in\mathcal{S}_h(t),\; \forall t\in[\tau_0,\tau_v)$, we define the following 6 candidate CBFs for $q\in\{x,y,z\}$:
\begin{equation} \label{CBF}
   \mat{h_{\overline{q}}\\h_{\underline{q}}} \triangleq\mat{\delta+ q^{\text{ref}}(t)-q(t)\\\delta+q(t)-q^{\text{ref}}(t)}.
\end{equation}
Note that all functions $h_{\overline{q}},h_{\underline{q}},\ q\in\{x,y,z\}$ have relative degree 2. 
According to Definition \ref{dfn:cbfhigh}, we choose constants $a_1,a_2$ such that  the roots of $\la^2+a_1\la+a_2$ are reals and negative, and the CBF conditions of relative degree 2 are satisfied for $h_{\overline{q}},h_{\underline{q}},\forall q\in\{x,y,z\}$.  For example,  the CBF condition for $h_{\overline{x}}$ is
\begin{align*}
L_g\bar L_fh_{\overline{x}}\vect{\mu} +\bar L_f^2h_{\overline{x}}+a_1\bar L_fh_{\overline{x}}+a_2h_{\overline{x}} \geq 0,    
\end{align*}
which is equivalent to
\begin{align}\label{cbfx}
\phi_{\overline{x}}^1\vect{\mu}+\phi_{\overline{x}}^2\leq 0,
\end{align}
where
\begin{align*}
    \phi_{\overline{x}}^1&=\vect{x}_W^T,\\
    \phi_{\overline{x}}^2&=-\ddot{x}^{\text{ref}}(t) + a_1(\dot{x}-\dot{x}^{\text{ref}}(t))+a_2(x -x^{\text{ref}}(t)-\delta).
\end{align*}
Putting 6 CBF conditions (for $h_{\overline{x}},h_{\underline{x}},h_{\overline{y}},h_{\underline{y}},h_{\overline{z}},h_{\underline{z}}$), which are all linear constraints on $\vect{\mu}$ for fixed $\vect{z}$, into a QP, and choosing the objective function such that  $\|\Psi(\pi(\vect{z})) - \vect{\mu}\|_2^2$ is minimized, where $\Psi(\pi(\vect{z}))$ yields the nominal virtual inputs, we have the following CBF-QP:
\begin{align} \label{CBF_QP}
\tag{CBF-QP}
    \min_{\vect{\mu}} \quad & \|\Psi(\pi(\vect{z})) - \vect{\mu}\|_2^2\\
    \text{s.t.} \quad &  \phi_{\overline{q}}^1\vect{\mu}+\phi_{\overline{q}}^2\leq 0,\nonumber\\ 
     &\phi_{\underline{q}}^1\vect{\mu}+\phi_{\underline{q}}^2\leq 0,\quad q\in\{x,y,z\},\nonumber
\end{align}
where the virtual input $\vect{\mu}$ is the decision variable, 
\begin{subequations}
\begin{align}
    \phi_{\overline{q}}^1&=\vect{q}_W^T,\label{phiup1}\\
    \phi_{\overline{q}}^2&=-\ddot{q}^{\text{ref}}(t) + a_1(\dot{q}-\dot{q}^{\text{ref}}(t))+a_2(q-q^{\text{ref}}(t)-\delta),\label{phiup2}\\
    \phi_{\underline{q}}^1&=-\vect{q}^T_W,\label{philow1}\\
    \phi_{\underline{q}}^2&=\ddot{q}^{\text{ref}}(t) + a_1(\dot{q}^{\text{ref}}(t)-\dot{q})+a_2(q^{\text{ref}}(t)-q-\delta),\label{philow2}
\end{align}
\end{subequations}

with $\vect{q}_W$ representing the appropriate world-frame's axis (see Figure \ref{fig:Coord_Frames}). Note that by the dynamic feasibility of the reference trajectory, we can extract $\ddot{q}^{\text{ref}}(t)$ from:
\begin{equation}
    \dot{\vect{z}}^{\text{ref}}(t) = A\vect{z}^{\text{ref}}(t) + B\Psi(\vect{v}^{\text{ref}}(t)), \quad t\in[\tau_0,\tau_v). 
\end{equation}

The \eqref{CBF_QP} is convex and can be solved in polynomial-time by QP solvers such as MOSEK \cite{mosek} or OSQP \cite{osqp}. Letting $\vect{\mu}^*(t)$ be the solution of \eqref{CBF_QP} for any $(t,\vect{z})\in[\tau_0,\tau_v)\times\mathbb{R}^6$, the safe input to the quadcopter is then  given as 
\begin{equation}\label{safe_input}
    \vect{v}_s(t) = \Psi^{-1}(\vect{\mu}^*(t),\psi(t)),
\end{equation}
where $\psi(t)$ is recovered from $\vect{v}(t) = \pi(\vect{z}(t))$.

The following result shows sufficient conditions that ensure the forward invariance of the set $\mathcal{S}_h(t)$, and feasibility of the \eqref{CBF_QP} with multiple CBF constraints.

\begin{theorem} \label{CBF_QP_Feasibility}
Consider the system given in \eqref{lin_sys}. If constants $a_1,a_2>0$ are chosen such that  the roots of $\la^2+a_1\la+a_2$ are negative reals $-\lambda_1,-\lambda_2<0$, 
then \eqref{CBF_QP} is  feasible for all $(t,\vect{z})\in[\tau_0,\tau_v)\times\mathbb{R}^6$. Furthermore, if 
$\|\vect{r}(\tau_0)-\vect{r}^{\text{ref}}(\tau_0) \|_{\infty} \leq \delta$, $\|\dot{\vect{r}}(\tau_0)-\dot{\vect{r}}^{\text{ref}}(\tau_0)+\lambda_1(\vect{r}(\tau_0)-\vect{r}^{\text{ref}}(\tau_0)) \|_{\infty} \leq \lambda_1\delta$,
and $\vect{\mu}^*(t)$ from \eqref{CBF_QP} is locally Lipchitz in $\vect{z}$, then the input $\vect{v}_s(t)$  will render $\|\vect{r}(t)-\vect{r}^{\text{ref}}(t) \|_{\infty} \leq \delta$ for all $t \in [\tau_0,\tau_v)$. Therefore, $\vect{z}(t)\in\mathcal{S}_h(t)$ for all $t\in[\tau_0,\tau_v)$.
\end{theorem}
\begin{proof}
The candidate CBFs defined in \eqref{CBF} aim to ensure $q^{\text{ref}}(t)-\delta \triangleq \underline q(t)\leq q(t)\leq \overline q(t)\triangleq q^{\text{ref}}+\delta$, for $q\in \{x,y,z\}$. Recalling that $\vect{\mu} =[\mu_1,\mu_2,\mu_3]^T$, the system model shown in \eqref{lin_sys} is equivalent to $\ddot x=\mu_1$, $\ddot y=\mu_2$, $\ddot z=\mu_3$. Thus, the axes are decoupled which enable us to consider each axis individually. It is easy to see that for $q\in \{x,y,z\}$, $\overline q-\underline q=2\delta$, $\dot{\overline q}-\dot{\underline q}=0$, and $\ddot{\overline q}-\ddot{\underline q}=0$. Since $a_1,a_2>0$, we have $\ddot{\overline q}-\ddot{\underline q}+a_1(\dot{\overline q}-\dot{\underline q})+a_2(\overline q-\underline q)=2a_2\delta>0$. Therefore, condition (15) in Theorem 2 of \cite{xu2018constrained} is satisfied, which implies that $h_{\overline{q}},h_{\underline{q}}$ are control sharing barrier functions. It is easy to see that  functions $h_{\overline{x}},h_{\underline{x}},h_{\overline{y}},h_{\underline{y}},h_{\overline{z}},h_{\underline{z}}$ are all control sharing barrier functions, which means that \eqref{CBF_QP} is  feasible for all $(t,\vect{z})\in[\tau_0,\tau_v)\times\mathbb{R}^6$. Since $\|\vect{r}(\tau_0)-\vect{r}^{\text{ref}}(\tau_0) \|_{\infty} \leq \delta$, $\|\dot{\vect{r}}(\tau_0)-\dot{\vect{r}}^{\text{ref}}(\tau_0)+\lambda_1(\vect{r}(\tau_0)-\vect{r}^{\text{ref}}(\tau_0)) \|_{\infty} \leq \lambda_1\delta$, we have $|q(\tau_0)-q^{\text{ref}}(\tau_0)|\leq\delta$, 
$|\dot{q}(\tau_0)-\dot{q}^{\text{ref}}(\tau_0)+\lambda_1q(\tau_0)-\lambda_1q^{\text{ref}}(\tau_0)|\leq \lambda_1\delta,\ \forall  q\in\{x,y,z\}$, which is equivalent to $h_{\overline{q}}(\tau_0)\geq 0$, $h_{\underline{q}}(\tau_0)\geq0$ and  $\dot h_{\overline{q}}(\tau_0)+\lambda_1 h_{\overline{q}}(\tau_0)\geq0$, $\dot h_{\underline{q}}(\tau_0)+\lambda_1h_{\underline{q}}(\tau_0)\geq0$ for $q\in\{x,y,z\}$. Therefore,  $h_{\overline{q}}(t)\geq 0$, $h_{\underline{q}}(t)\geq 0$ for all $t \in [\tau_0,\tau_v)$ and $q\in\{x,y,z\}$ by \cite{xu2018constrained}, which implies that  $\|\vect{r}(t)-\vect{r}^{\text{ref}}(t) \|_{\infty} \leq \delta$ holds for $t \in [\tau_0,\tau_v)$. The final conclusion follows by the definition of set $\mathcal{S}_h(t)$.
\end{proof}



\begin{remark}
The feasibility of \eqref{CBF_QP} can be also seen from the non-emptiness of its admissible input set. 
For $q\in\{x,y,z\}$, define the respective admissible input sets corresponding to CBFs $h_{\overline{q}}$ and $h_{\underline{q}}$ as $\mathcal{K}_{\overline{q}}=\{\vect{\mu}\mid \phi_{\overline{q}}^1\vect{\mu}+\phi_{\overline{q}}^2\leq 0\},\mathcal{K}_{\underline{q}}=\{\vect{\mu}\mid \phi_{\underline{q}}^1\vect{\mu}+\phi_{\underline{q}}^2\leq 0\}$, respectively.  
Note that $-\phi_{\overline{q}}^{2}\geq\phi_{\underline{q}}^2$ holds $\forall q\in\{x,y,z\}$ because $-\phi_{\overline{q}}^2 - \phi_{\underline{q}}^2=2a_2\delta>0$. Since $\mathcal{K}_{x}\triangleq\mathcal{K}_{\overline{x}}\cap \mathcal{K}_{\underline{x}}=\{\vect{\mu}\mid \phi_{\underline{x}}^2\leq \mu_1\leq -\phi_{\overline{x}}^2\}$, we have $\mathcal{K}_{x}\neq\emptyset$. Similarly, 
$\mathcal{K}_{y}\triangleq\mathcal{K}_{\overline{y}}\cup \mathcal{K}_{\underline{y}}=\{\vect{\mu}\mid \phi_{\underline{y}}^2\leq \mu_2\leq -\phi_{\overline{y}}^2\}\neq\emptyset$, and
$\mathcal{K}_{z}\triangleq\mathcal{K}_{\overline{z}}\cup \mathcal{K}_{\underline{z}}=\{\vect{\mu}\mid \phi_{\underline{z}}^2\leq \mu_3\leq -\phi_{\overline{z}}^2\}\neq\emptyset$. It is easy to see that $\mathcal{K}_{x}\cap\mathcal{K}_{y}\cap\mathcal{K}_{z}\neq\emptyset$ as
 $\mathcal{K}_{x},\mathcal{K}_{y},\mathcal{K}_{z}$ impose constraints on $\mu_1,\mu_2,\mu_3$, respectively. Therefore, the admissible input set  of \eqref{CBF_QP} is always non-empty, which guarantees the feasibility of \eqref{CBF_QP}.
\end{remark}

\begin{remark} \label{sdcbf}
The safety guarantee provided by Theorem \ref{CBF_QP_Feasibility} is predicated on the assumption that the control input generated by the \eqref{CBF_QP} can be updated continuously. When a sampled-data controller is implemented in practice, it is still possible to provide the same safety guarantee in continuous time, if we make a slight modification to the CBF conditions of \eqref{CBF_QP}. We refer the readers to our recent work \cite{zhang21interval} that proposes  
sampled-data CBFs suitable for real-time computation, and demonstrates the related algorithms using experiments in real hardware.
\end{remark}

For the trajectory obtained from \eqref{CBF_QP}, we can also quantify the upper bounds of $\|\dot{\vect{r}}(t) -\dot{\vect{r}}^{\text{ref}}(t)\|_{\infty}$ and $\|\vect{\mu} - \ddot{\vect{r}}^{\text{ref}}(t)\|_{\infty}$ as shown in the following two corollaries.

\begin{corollary} \label{bounded_vel_corollary}
If the conditions of Theorem \ref{CBF_QP_Feasibility} are satisfied and $\|\dot{\vect{r}}(\tau_0) -\dot{\vect{r}}^{\text{ref}}(\tau_0)\|_{\infty} \leq \frac{2\delta a_2}{a_1}$, then
\begin{equation} \label{cbf_vel_bounds}
    \|\dot{\vect{r}}(t) -\dot{\vect{r}}^{\text{ref}}(t)\|_{\infty} \leq \frac{2\delta a_2}{a_1}
\end{equation}
holds for any $t\in[\tau_0,\tau_v)$. 
\end{corollary}
\begin{proof}
Any feasible solution $\vect{\mu} = (\ddot{x},\ddot{y},\ddot{z})^T$ of \eqref{CBF_QP} satisfies:
\begin{subequations}
\begin{align}
    \ddot{q}&\leq \ddot{q}^{\text{ref}}+a_1(\dot{q}^{\text{ref}}-\dot{q})+a_2(q^{\text{ref}}-q+\delta),\label{veleq1}\\
    \ddot{q}&\geq \ddot{q}^{\text{ref}} + a_1(\dot{q}^{\text{ref}}-\dot{q}) + a_2(q^{\text{ref}} -q - \delta),\label{veleq2}
\end{align}
\end{subequations}
for $q\in\{x,y,z\}$. Since $|q-q^{\text{ref}}|\leq \delta$, \eqref{veleq1}-\eqref{veleq2} are equivalent to:
\begin{subequations}
\begin{align}
    \ddot{q}&\leq \ddot{q}^{\text{ref}}+a_1(\dot{q}^{\text{ref}}-\dot{q})+2a_2\delta,\label{veleq3}\\
    \ddot{q}&\geq \ddot{q}^{\text{ref}} + a_1(\dot{q}^{\text{ref}}-\dot{q}) - 2a_2\delta.\label{veleq4}
\end{align}
\end{subequations}
Define candidate CBFs $\hat{h}_{\overline{q}}$ and $\hat{h}_{\underline{q}}$ as follows:
\begin{align*}
\hat{h}_{\overline{q}}(t)&\triangleq \dot{q}^{\text{ref}}(t)-\dot{q}(t)+ \frac{2\delta a_2}{a_1},\\ \hat{h}_{\underline{q}}(t)&\triangleq \dot{q}(t) - \dot{q}^{\text{ref}}(t) + \frac{2\delta a_2}{a_1}.
\end{align*}
Then \eqref{veleq3}-\eqref{veleq4} are equivalent to 
$$
\dot{\hat{h}}_{\overline{q}}(t)+a_1\hat{h}_{\overline{q}}(t)\geq 0,\quad \dot{\hat{h}}_{\underline{q}}(t)+a_1\hat{h}_{\underline{q}}(t)\geq 0,
$$
for $q\in\{x,y,z\}$ and any $t\in[\tau_0,\tau_v)$. Therefore,  $\hat{h}_{\overline{q}}$ and $\hat{h}_{\underline{q}}$ satisfy the CBF condition. Furthermore, 
since $\|\dot{\vect{r}}(\tau_0) -\dot{\vect{r}}^{\text{ref}}(\tau_0)\|_{\infty} \leq \frac{2\delta a_2}{a_1}$, we have $\hat{h}_{\overline{q}}(\tau_0)\geq 0$ and $\hat{h}_{\underline{q}}(\tau_0)\geq 0$ for $q\in\{x,y,z\}$. Thus, $\hat{h}_{\overline{q}}(t)\geq 0$ and $\hat{h}_{\underline{q}}(t)\geq 0$ for $t\in[\tau_0,\tau_v)$ by \cite{ames2016control}, which implies that $|\dot{q}^{\text{ref}}(t)-\dot{q}(t)|\leq \frac{2\delta a_2}{a_1}$ for $q\in\{x,y,z\}$ and $t\in[\tau_0,\tau_v)$. The conclusion \eqref{cbf_vel_bounds} follows immediately.
\end{proof}

\begin{corollary} \label{cbf_qp_feas_set_bound_cor}
If the conditions of Theorem \ref{CBF_QP_Feasibility} are satisfied, then the feasible solution $\vect{\mu}$ of \eqref{CBF_QP} is bounded by:
\begin{equation}\label{cbf_qp_feas_set_bound}
    \|\vect{\mu} - \ddot{\vect{r}}^{\text{ref}}(t)\|_{\infty}\leq 4\delta a_2,
\end{equation}
for any $t\in[\tau_0,\tau_v)$.
\end{corollary}
\begin{proof}
Recall that $\|\vect{r}(t)-\vect{r}^{\text{ref}}(t)\|_{\infty}\leq \delta$ by Theorem \ref{CBF_QP_Feasibility} and $\|\dot{\vect{r}}(t) -\dot{\vect{r}}^{\text{ref}}(t)\|_{\infty} \leq 2\delta a_2/a_1$ by Corollary \ref{bounded_vel_corollary}.
These bounds allow us to upper and lower bound the feasible set of \eqref{CBF_QP} by:
\begin{equation*}
    \ddot{q}^{\text{ref}}(t) - 4 \delta a_2\leq \ddot{q}\leq \ddot{q}^{\text{ref}}(t)+4\delta a_2,
\end{equation*}
where $q\in\{x,y,z\}$. Noticing that $\vect{\mu} = (\ddot{x},\ddot{y},\ddot{z})^T$ concludes the proof.
\end{proof}


\subsection{Input Constraints Satisfaction for the CBF-QP Controller}
The thrust $T$, the roll angle $\phi$, and the pitch angle $\theta$ are related to the decision variable $\vect{\mu}$
in \eqref{CBF_QP} through the relation shown in \eqref{safe_input}. In this subsection, we will consider how the safety constraints on $T,\phi,\theta$ considered in Section \ref{sec:traj} can be respected by the CBF-QP-based tracking controller that filters the nominal controller $\pi(\vect{z})$.
Specifically,  we aim to guarantee that $\vect{v}_s(t)\in\mathcal{V}$ where $\vect{v}_s(t)$ is given in \eqref{safe_input} and
\begin{subequations}
\begin{align}
    \mathcal{V}&\triangleq \mathcal{V}_{T}\cap \mathcal{V}_{\xi}, \label{eqV}\\
 \mathcal{V}_{T} &\triangleq  \{\vect{v}\in\mathbb{R}^4\mid 0\leq T\leq \overline{T}\},\\
    \mathcal{V}_{\xi}&\triangleq \{\vect{v}\in\mathbb{R}^4: |\phi|,|\theta|\leq \epsilon<\pi/2\},
\end{align}
\end{subequations}
with $\epsilon$ the upper bound for $|\phi|$ and $|\theta|$, and $\overline{T}$ the upper bound for $T$. Note that the constraint $\vect{v}_s(t)\in\mathcal{V}$ is important to ensure actuator saturation will not happen.
\begin{proposition} \label{input_safeguard_prop}
Consider the system given in \eqref{lin_sys} and suppose that the conditions of Theorem \ref{CBF_QP_Feasibility} are satisfied. If trajectories of the state $\vect{z}^{\text{ref}}(t)$ and the input $\vect{v}^{\text{ref}}(t)$ are dynamically feasible by satisfying \eqref{dyn_model_acc}  and 
the following conditions hold for all $t\in[\tau_0,\tau_v)$:
\begin{subequations} \label{sat_safeguard_cst}
\begin{align}
0\leq T^{\text{ref}}(t) &\leq \overline{T} - 4\sqrt{3}\delta a_2,\label{satprop_T}\\
\|A_{\epsilon}\ddot{\vect{r}}^{\text{ref}}(t)\|_2&\leq \ddot{z}^{\text{ref}}(t) +g -4\delta a_2(1+|\cot\epsilon|\sqrt{2}),\label{satprop_eul}
\end{align}
\end{subequations}
where $A_{\epsilon} =|\cot\epsilon|\mat{\vect{x}_W & \vect{y}_W & \vect{0}_{3\times 1}}$, then $\vect{v}_s(t)\in\mathcal{V}$ with $\vect{v}_s(t)$ given in \eqref{safe_input} and $\mathcal{V}$ given in \eqref{eqV}.
\end{proposition}
\begin{proof} 
Let $\vect{\mu}^*(t)$ be the solution of \eqref{CBF_QP}. Note that by norm equivalence, \eqref{cbf_qp_feas_set_bound} implies:
\begin{equation}
    \|\vect{\mu}^*(t) - \ddot{\vect{r}}^{\text{ref}}(t)\|_2\leq \sqrt{3}\|\vect{\mu}^*(t)-\ddot{\vect{r}}^{\text{ref}}(t)\|_\infty\leq4\sqrt{3}\delta a_2,\label{eqnprp3eq1}
\end{equation}
for any $t\in[\tau_0,\tau_v)$. 

First, consider the thrust element of $\vect{v}_s(t)$ given as $T(t) = \|\vect{\mu}^*(t)+g\vect{z}_W\|_2$ by the map $\Psi^{-1}$. By $\eqref{satprop_T}$ we have
\begin{equation}
    T^{\text{ref}}(t)+4\sqrt{3}\delta a_2 = \|\ddot{\vect{r}}^{\text{ref}}(t)+g\vect{z}_W\|_2+4\sqrt{3}\delta a_2 \leq \overline{T}.\label{eqnprp3eq2}
\end{equation}
Therefore, 
\begin{align*}
T(t) - \overline{T} &=\|\vect{\mu}^*(t)+g\vect{z}_W\|_2-\overline{T}\\
    &\leq \|\vect{\mu}^*(t)+g\vect{z}_W\|_2-\|\ddot{\vect{r}}^{\text{ref}}(t)+g\vect{z}_W\|_2-4\sqrt{3}\delta a_2\\
    &\leq \|\vect{\mu}^*(t)-\ddot{\vect{r}}^{\text{ref}}(t)\|_2-4\sqrt{3}\delta a_2 \\
    &\leq 0,
\end{align*}
where the first inequality is from \eqref{eqnprp3eq2}, the second inequality is from the reverse triangle inequality, and the third inequality is from \eqref{eqnprp3eq1}. Thus, $T(t)\leq \overline{T}$ for $t\in[\tau_0,\tau_v)$. Note that $\ddot{\vect{r}}^{\text{ref}}(t)$ can be extracted from $\dot{\vect{z}}^{\text{ref}}(t)$ which is given by dynamic feasibility of the reference trajectory. 

Next, we examine the roll $\phi(t)$ and pitch $\theta(t)$ elements of $\vect{v}_s(t)$. We drop the time-dependence of terms for convenience. Let $\mathcal{K}_{\epsilon}$ be given as in \eqref{ang_bounds_cone}. 
Note that
\begin{align}
\lv\mat{\mu^*_1\\\mu^*_2}\rv_2&\leq \lv\mat{\mu^*_1\\\mu^*_2}-\mat{\ddot x^{\text{ref}}\\\ddot y^{\text{ref}}}\rv_{2}+\lv\mat{\ddot x^{\text{ref}}\\\ddot y^{\text{ref}}}\rv_2\nonumber\\
&\leq 4\sqrt{2}\delta a_2+\lv\mat{\ddot x^{\text{ref}}\\\ddot y^{\text{ref}}}\rv_2\label{ineqAep}
\end{align}
where the first inequality is from the 
triangle inequality and the second inequality is from norm equivalence and Corollary \ref{cbf_qp_feas_set_bound_cor}.
Therefore,
\begin{align*}
\|A_{\epsilon}\vect{\mu}^*\|_2 =|\cot\epsilon|\lv\mat{\mu^*_1\\\mu^*_2}\rv_2
&\leq 
4|\cot\epsilon|\sqrt{2}\delta a_2+ \|A_{\epsilon}\ddot{\vect{r}}^{\text{ref}}\|_2  \\
&\leq\ddot{z}^{\text{ref}} -4\delta a_2 +g \\
&\leq \mu_3^*+g
\end{align*}
where the first inequality is from \eqref{ineqAep} and the definition of $A_{\epsilon}$, the second inequality is from \eqref{satprop_eul}, and the third inequality is from the fact that $\ddot{z}^{\text{ref}} \leq \mu_3^*+4\delta a_2$, which can be derived from  $|\mu_3^* - \ddot{z}^{\text{ref}}|\leq 4\delta a_2$ by Corollary \ref{cbf_qp_feas_set_bound_cor}. 
Thus, $\mu^*(t)\in\mathcal{K}_{\epsilon}$ and, by Lemma \ref{epsilon_cone_lemma}, 
we have that $|\phi(t)|,|\theta(t)|\leq \epsilon$ for all $t\in[\tau_0,\tau_v)$. This completes the proof.
\end{proof}

\begin{remark}
The conditions \eqref{satprop_T}-\eqref{satprop_eul}  for the reference trajectory can be easily guaranteed in continuous-time with the method presented in Section \ref{sec:traj}, and the resulting \eqref{FLAT-SOCP} remains convex. Intuitively, conditions \eqref{satprop_T}-\eqref{satprop_eul} shrink the cone \eqref{ang_bounds_cone} and the sphere \eqref{T_sphere} constraints on the trajectory acceleration (see Figure \ref{fig:satguard}) enough so that all feasible solutions of \eqref{CBF_QP} satisfy the desired safety constraints when transformed through the map $\Psi^{-1}$. This means that the transformed optimal solution $\vect{v}_s(t) = \Psi^{-1}(\vect{\mu}^*(t),\psi) \in \mathcal{V}$ for any $\psi$ and any $t\in[\tau_0,\tau_v)$.
\begin{figure}[htp]
    \centering
    \includegraphics[width=8cm]{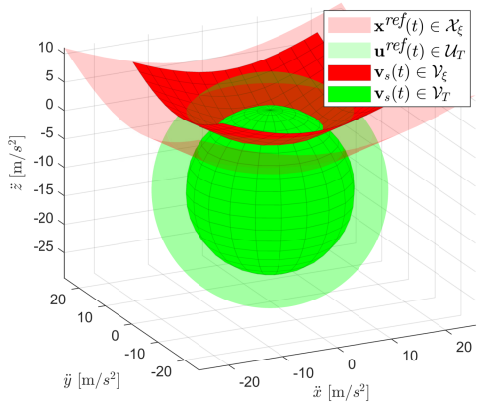}
    \caption{
    Conditions \eqref{satprop_T}-\eqref{satprop_eul}  shrink the cone that bounds the Euler angles ($\vect{x}^{\text{ref}}(t)\in\mathcal{X}_{\xi}$) and the sphere that bounds the thrust ($\vect{u}^{\text{ref}}(t)\in\mathcal{U}_{T}$) to guarantee that safe inputs $\vect{v}_{s}(t)\in\mathcal{V}$. This figure shows the shrinkage with typical bounds $\epsilon = 1.0472$, $\overline{T}=2g$ and the same safety parameters  $\delta=0.1$, $a_2=8$ that we use in Examples \ref{example:vcp}, \ref{example:slowzone} and \ref{example:forest} in Sec. \ref{sec:exper}}
    \label{fig:satguard}
\end{figure}
\end{remark}

\begin{remark}\label{remark12}
Instead of  shrinking the cone  and the sphere  constraints for \eqref{FLAT-SOCP}, 
an alternative method to ensuring $\vect{v}_s(t)\in\mathcal{V}$ is to convert the CBF-based QP into a SOCP by adding two additional SOC constraints as follows:
\begin{align*} 
    \min_{\vect{\mu}} \quad & \|\Psi(\pi(\vect{z})) - \vect{\mu}\|_2^2\\
    \text{s.t.} \quad &  \phi_{\overline{q}}^1\vect{\mu}+\phi_{\overline{q}}^2\leq 0,\nonumber\\ 
     &\phi_{\underline{q}}^1\vect{\mu}+\phi_{\underline{q}}^2\leq 0,\quad q\in\{x,y,z\},\nonumber\\
     &\|\vect{\mu}+g\vect{z}_W\|_2\leq \overline{T},\\
     &\|A_{\epsilon}\vect{\mu}\|_2\leq \mu_3+g,
\end{align*}
with $A_{\epsilon}$ given in Proposition \ref{input_safeguard_prop}. Note that the optimization shown above is still a convex program, but its feasibility is no longer guaranteed.


\end{remark}

\section{Experiments}\label{sec:exper}
We demonstrate the effectiveness and versatility of the proposed framework using three experimental examples. Each example highlights certain different aspects of the framework. For all examples, \eqref{FLAT-SOCP} is solved in MATLAB with Yalmip \cite{Lofberg2004} and solver MOSEK \cite{mosek} at a rate of 30-100 Hz depending on its complexity, and \eqref{CBF_QP} is solved in Python3 with OSQP \cite{osqp} at a rate of 100 Hz. The tracking data is obtained from an OptiTrack motion capture system and a Crazyflie2.1 nano quadcopter. We provide videos of all presented examples in \cite{victor21video}.


\begin{example}\label{example:vcp}

The first example demonstrates rigorous satisfaction of the safety constraints in continuous-time by the reference trajectory obtained from \eqref{FLAT-SOCP} and the bounded tracking performance provided by the CBF-QP controller. 
First, we solve \eqref{FLAT-SOCP} to generate a reference trajectory satisfying \eqref{state_input_const} where the following safety constraints are required to be respected for all $t\in[\tau_0,\tau_v)$:
\begin{align*}
    \|\dot{\vect{r}}^{\text{ref}}(t)\|&\leq \overline{v},\\
    |\phi^{\text{ref}}(t)|,|\theta^{\text{ref}}(t)|&\leq \epsilon,\\
    \underline{T}\leq T^{\text{ref}}(t)&\leq \overline{T},\\
    |p^{\text{ref}}(t)|,|q^{\text{ref}}(t)|&\leq \overline{\omega}.
\end{align*}
The parameters of the trajectory and the safety cosntraints are given in Table \ref{tab:sample_trajectory_1}, and the parameters of the waypoints are given in Table \ref{tab:wp_1}. The initial/final conditions are chosen as $\vect{p}^m_0=\cdots =\vect{p}^m_4= (0,0,0)^T, m\in\{0,f\}$.   
With these parameters, we solve \eqref{FLAT-SOCP}  and show the resulting position trajectory $\vect{r}^{\text{ref}}(t)$ in Figure \ref{fig:vcp_st_r0} along with the waypoints and control points. The geometry of the constraints enforced for the $2$nd order VCPs $\vect{P}_j^{(2)}$ is shown in Figure \ref{fig:vcp_st_r2}:
\begin{itemize}
    \item Inclusion in the conic surface (red) given by $\mathcal{K}_{\epsilon}$ defined in \eqref{ang_bounds_cone} to guarantee $\vect{x}^{\text{ref}}(t)\in \mathcal{X}_{\xi}$.
    \item Inclusion in the sphere (green) given by \eqref{vcp_Tmax} to guarantee upper-bounded thrust $T(t)\leq \overline{T}$.
    \item Restriction above hyperplane (cyan) given by \eqref{vcp_Tmin} to guarantee lower-bounded thrust $\underline{T}\leq T(t)$.
\end{itemize}
The safety constraint satisfaction is verified in Figure \ref{fig:vcp_st_cst} by transforming the flat output trajectory $\vect{\sigma}^{\text{ref}}(t)=(\vect{r}^{\text{ref}}(t),0) $ via \eqref{flat_map} to the state $\vect{x}^{\text{ref}}(t)$ and input $\vect{u}^{\text{ref}}(t)$ trajectories. From  Figure \ref{fig:vcp_st_cst}, it can be observed that all considered safety constraints are satisfied in continuous time. It can be also seen that the values of the vector $\vect{\zeta}=(\zeta_1,...,\zeta_{36})^T$, which is illustrated by the light red line, lower-bound the thrust over local segments as defined in \eqref{cp_omega}. 

\begin{table} [bt]
\renewcommand{\arraystretch}{1.3}
\caption{Trajectory parameters for Example \ref{example:vcp}}
\label{tab:sample_trajectory_1}
\centering
\begin{tabular}{|c|c||c|c||c|c|}
\hline
$\tau_0$ [s] & 0 &  $n_0, n_f$ [-] & 4 & $\epsilon$ [deg] & 1.75 \\
\hline
$\tau_v$ [s] & 30 & $n_{wp}$ & 8 &  $\underline{T}$ $[m/s^2]$ & 9.7 \\
\hline
N [-] & 40 & $d_{wp}$ [m]& 0.05 & $\overline{T}$  $[m/s^2]$ & 9.9 \\
\hline
d [-] & 5  & $\overline{v}$ $[m/s]$ & 0.5 & $\overline{\omega}$ [deg/s] & 1.5\\
\hline
\end{tabular}
\end{table}

\begin{table} [bt]
\renewcommand{\arraystretch}{1.3}
\caption{Trajectory waypoints for Example \ref{example:vcp}}
\label{tab:wp_1}
\centering
\begin{tabular}{|c|c c c c c c c c|}
\hline
$i$ & 1 & 2 & 3 & 4 & 5 & 6 & 7 & 8\\
\hline
& -0.15 & -0.75 & 0.65 & 0.65 & -0.5 & -0.6 & 0.4 & 0.25\\
$\vect{p}_i^{wp}$ & 0.25 & 0.6 & -0.65 & 0.5 & 0.5 & -0.6 & -0.4 & 0.25\\
& 0.25 & 0.5 & 0.25 & 0.25 & 0.75 & 0.5 & 0.4 & 0.25\\
\hline
$t_i$ & 4.5 & 7.8 & 12.6 & 15.3 & 18 & 21  & 24 & 27\\
\hline
\end{tabular}
\end{table}

\begin{figure}[htp]
    \centering
    \includegraphics[width=8cm]{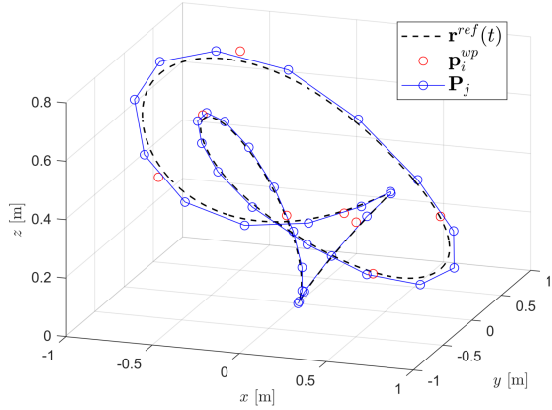}
    \caption{The position trajectory generated in Example \ref{example:vcp} with the parameters in Tables \ref{tab:sample_trajectory_1} and \ref{tab:wp_1}. We show the position trajectory (dashed black) along with the waypoints (red circles) and the control points (connected, blue circles).}
    \label{fig:vcp_st_r0}
\end{figure}

\begin{figure}[htp]
    \centering
    \includegraphics[width=8cm]{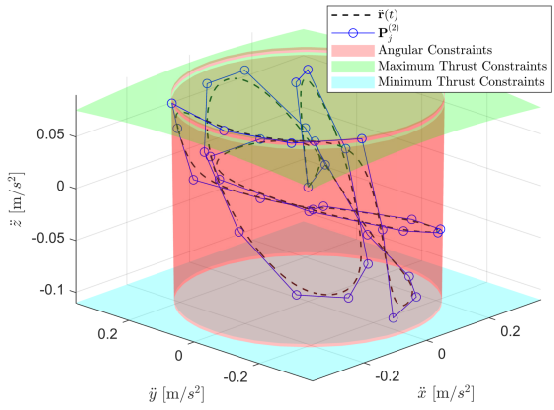}
    \caption{
    Visualization of constraints enforced for the $2$nd order VCPs $\vect{P}_j^{(2)}$ in Example \ref{example:vcp}. Constraint \eqref{cp_ang} represents inclusion within the red cone. Constraint \eqref{cp_Tmax} is the inclusion within the green sphere (only cusp is shown). Constraint \eqref{cp_Tmin} forces the VCPs to lie above the cyan plane (global bound). Note that the local counterpart \eqref{cp_omega_Tmin} is not depicted here. Note also that these are not the only constraints added to \eqref{FLAT-SOCP} in Example \ref{example:vcp}.}
    \label{fig:vcp_st_r2}
\end{figure}

\begin{figure}[htp]
    \centering
    \subfigure[State constraints. The velocity respects $\|\dot{\vect{r}}(t)\|\leq 0.5$ m/s and the Euler angles respect $|\phi(t)|,|\theta(t)|\leq 1.75$ deg for all $t\in[0,30)$.]{
    \includegraphics[width=8cm]{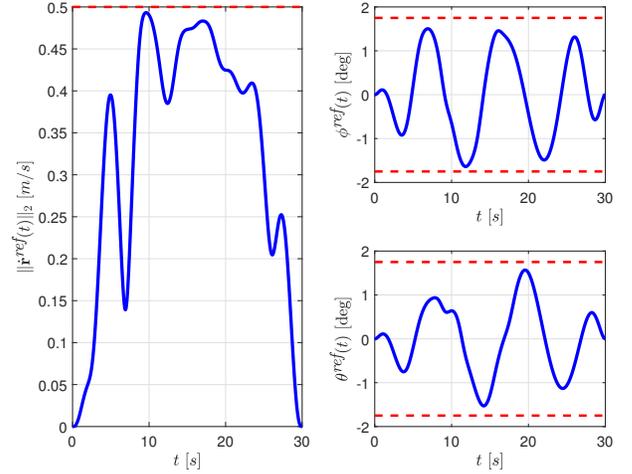}
    }
    \subfigure[Input constraints. The thrust respects $9.7\;  \text{m/$s^2$}\leq T(t) \leq 9.9\;  \text{m/$s^2$}$ and the angular velocities respect $|p(t)|,|q(t)|\leq 1.5$ deg/s. The light red line indicates the values of the vector $\vect{\zeta}=(\zeta_1,\ldots,\zeta_{36})^T$ lower-bounding the thrust over local segments as shown in \eqref{cp_omega}.]{
    \includegraphics[width=8cm]{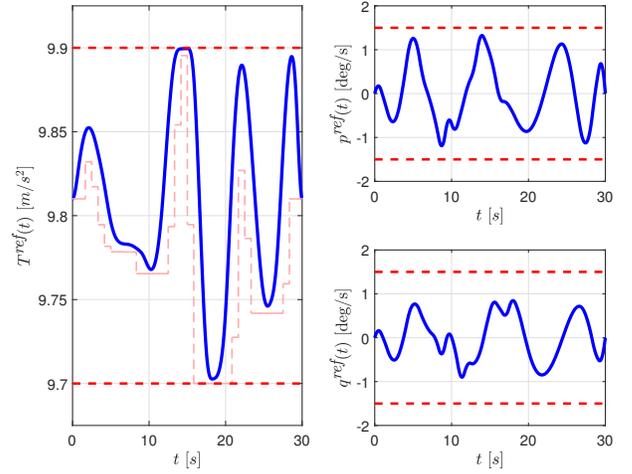}
    }
    \caption{Verification that the safety constraints for state and inputs in Example \ref{example:vcp} are satisfied in continuous-time.}
    \label{fig:vcp_st_cst}
\end{figure}

\begin{figure}[htb]
    \centering
    \subfigure[Tracking performance of the nominal controller $\pi(\vect{z})$ with (blue) and without (magenta) the \eqref{CBF_QP}. The data is obtained from a real flight experiment using a Crazyflie2.1 nano quadcopter.]{
    \includegraphics[width=8cm]{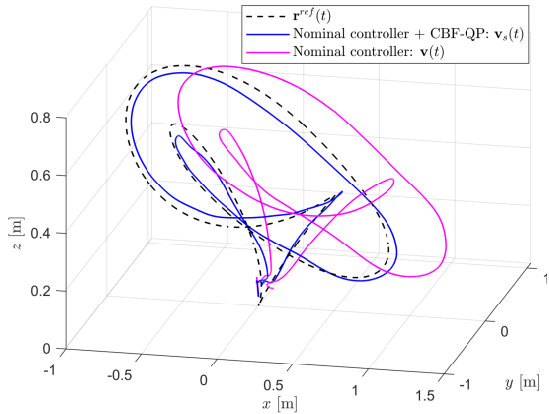}
    }
    \subfigure[Measured values of  $h_{\overline{x}},h_{\underline{x}},h_{\overline{y}},h_{\underline{y}},h_{\overline{z}},h_{\underline{z}}$ defined in \eqref{CBF} when tracking with the filtered (blue) and unfiltered (magenta) nominal controller $\pi(\vect{z})$. Note that safety is satisfied when the values of CBFs are nonnegative, which means that the real trajectory is within the prescribed tube ($\delta =0.1$) around the nominal  trajectory. Tracking with unfiltered nominal controller (magenta) violates safety as $h_{\overline{x}}(t,\vect{z}(t)) <0$ and $h_{\overline{y}}(t,\vect{z}(t)) <0$ for some $t$. On the other hand, the CBF-QP tracking controller always respects safety as $h_{\overline{x}},h_{\underline{x}},h_{\overline{y}},h_{\underline{y}},h_{\overline{z}},h_{\underline{z}}$ are always non-negative, indicating safety. ]{\includegraphics[width=8cm]{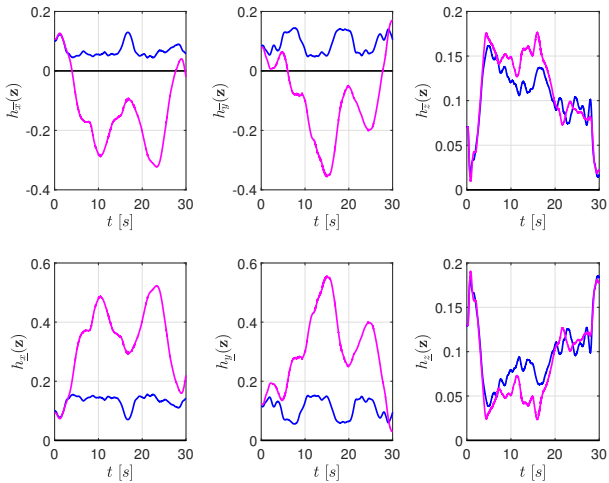}}
    \caption{Verification of the bounded-tracking property of \eqref{CBF_QP} with experimental data for Example \ref{example:vcp}. 
    }
    \label{fig:cbf_qp}
\end{figure}

Next, we implement in a Crazyflie2.1 nano quadcopter the CBF-QP controller to track the reference trajectory generated by \eqref{FLAT-SOCP} with a poorly-designed nominal controller $\pi(\vect{z})$. We show that, when filtered through the \eqref{CBF_QP} with parameters $a_1 = 6, a_2 =8$, safe trajectory tracking can be achieved with $\delta=0.1$. We compare the performance of the CBF-QP tracking controller and the unfiltered nominal controller in Figure \ref{fig:cbf_qp}. It can be seen that the CBF-QP controller  achieves much better tracking performance than the unfiltered nominal controller. Note that the bounding constraints are satisfied by the  CBF-QP controller, which is demonstrated by the fact that CBFs defined in \eqref{CBF} satisfy $h_{\overline{q}},h_{\underline{q}}\geq 0$ for $q\in\{x,y,z\}$, while these conditions are not true for the unfiltered nominal controller.   

\end{example}

\begin{example}\label{example:slowzone}
The second example  demonstrates the proposed framework is suitable for real-time (online) computation.
We let the quadcopter take off from, and land on a moving platform, and fly through a velocity-constrained zone while passing two nearby waypoints. Note that velocity-constrained areas are commonplace  in  safety-critical applications (e.g., when the quadcopter operates near humans). The reference trajectory is recomputed by solving \eqref{FLAT-SOCP} in real time, where the initial and final positions of the reference trajectory are both set to be the position of the moving platform. The tracking controller is obtained by solving \eqref{CBF_QP} online.

The safety specifications for the reference trajectory are: 
\begin{subequations} \label{slowzone_specs}
\begin{align}
    \vect{r}^{\text{ref}}(t) \in \mathcal{S}_v, \quad
    \|\dot{\vect{r}}^{\text{ref}}(t)\| \leq \overline{v} = 0.5, \quad t\in[3,6),\label{slowzone_specs_local}\\
    \vect{r}^{\text{ref}}(0) =  \vect{r}^{\text{ref}}(9) = \vect{p}_{0f}(t), \quad \dot{\vect{r}}^{\text{ref}}(0) = \dot{\vect{r}}^{\text{ref}}(9) = 0,\label{slowzone_specs_ext} \\
    \|\vect{r}^{\text{ref}}(2.5) - \vect{p}_1^{wp}\|_2\leq d_{wp},\|\vect{r}^{\text{ref}}(6.5) - \vect{p}_2^{wp}\|_2\leq d_{wp}\label{slowzone_specs_wp},
\end{align}
\end{subequations}
where $\vect{p}_{0f}(t)$ is the position of the platform at time $t$, $\vect{p}_1^{wp} = (0.75, 0.6, 1.1)^T$ and $\vect{p}_2^{wp} = (-0.75,0.6,1.1)^T$ are the waypoints with desired radius $d_{wp} = 0.2$, and the constrained-velocity zone $\mathcal{S}_v$ is defined as: 
\begin{align*}
    \mathcal{S}_{v}&\triangleq \{\vect{r}\in\mathbb{R}^3: \|A_s\vect{r}+\vect{b}_s\|_2\leq 1\}, \label{ex2_local_specification} \\
    A_s&=\text{diag}(1.33, 13.3, 13.3),\\
    \vect{b}_s &= \mat{0& -10& -14.7}^T,
\end{align*}
which describes an ellipsoid centered at $\vect{p}_h = (0,0.75,1.1)^T$ chosen to ensure the quadcopter clears the \emph{hoop} obstacle.

%

We begin by considering a B-spline curve $\vect{r}^{\text{ref}}(t)$ as defined in \eqref{b-spline_curve} with $N= 45$, $d=5$, $\tau_0 = 0$ and $\tau_v = 9$. Including \eqref{slowzone_specs_ext} and \eqref{slowzone_specs_wp} in \eqref{FLAT-SOCP} is straightforward. However, \eqref{slowzone_specs_local} indicates position and velocity bounds over a local time interval and will result in constraints over a subset of the control points and VCPs.
Since $\vect{\tau}$ is clamped and uniform, we have that $[3,6)\in[\tau_{18},\tau_{33}) = [2.85, 6.15)$. Using Proposition \ref{inclusion_prop}, we obtain the following constraints on the control points and $1st$ order VCPs:
\begin{align*}
  &\vect{P}_j \in S_v, \quad j=13,\dots, 32,\\
  \|&\vect{P}_j^{(1)}\|_2\leq \overline{v}, \quad j=14,\dots,32.
\end{align*}
In the experiment, measurements of $\vect{p}_{0f}(t)$ are obtained at 100 Hz and the convexity of \eqref{FLAT-SOCP} allows us to recompute trajectories $\vect{r}^{\text{ref}}(t)$ satisfying \eqref{slowzone_specs} in real-time. The tracking is done by a nominal controller $\pi(\vect{z})$ filtered by \eqref{CBF_QP} with parameters $\delta = 0.1$, $a_1 = 6$ and $a_2 = 8$.

We implement this experiment using a Crazyflie2.1 nano quadcopter and a Turtlebot 2 as the dynamic platform. Figure \ref{fig:slowzone_3D} shows the considered scenario, the quadcopter tracking data and the measurements of $\vect{p}_{0f}(t)$.
The Crazyflie2.1 takes off near $\vect{p}_{0f}(0)$ and lands near $\vect{p}_{0f}(9)$, by tracking the reference trajectory $\vect{r}^{\text{ref}}(t)$, which is recomputed at 30Hz. Additionally, we show the velocity profiles $\|\dot{\vect{r}}^{\text{ref}}(t)\|_2$ of the reference trajectories in Figure \ref{fig:slowzone_vel}. It can be seen that all computed trajectories respect the constrained-velocity safety specification $\|\dot{\vect{r}}^{\text{ref}}(t)\|_2\leq 0.5$ for all $t\in[3,6)$. We also show the measured velocity magnitude of the quadcopter throughout the experiment.


\begin{figure}
    \centering
    \subfigure[\label{fig:slowzone_3D}Scenario for Example \ref{example:slowzone} showing the waypoints $\vect{p}^{wp}_1$ and $\vect{p}^{wp}_2$ with radius $d_{wp}=0.2$, the constrained-velocity zone $\mathcal{S}_v$ through the \emph{hoop} obstacle, and the measured trajectories of the dynamic platform (i.e., $\vect{p}_{0f}(t)$) and  Crazyflie2.1. 
    The quadcopter takes off near $\vect{p}_{0f}(0)$ and lands near $\vect{p}_{0f}(9)$, which is accomplished by recomputing the reference trajectory $\vect{r}^{\text{ref}}(t)$. 
    ]{\includegraphics[width=0.4\textwidth]{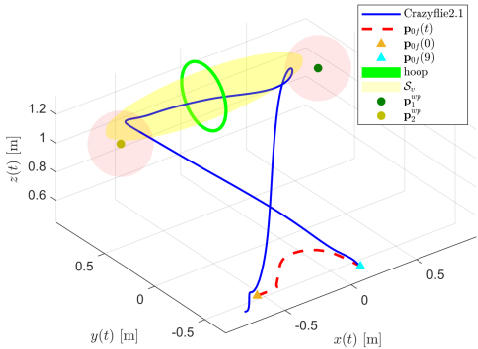}}
    \subfigure[\label{fig:slowzone_vel}Velocity profiles of the trajectories in Example \ref{example:slowzone}. Each reference trajectory satisfies the safety specification $\|\dot{\vect{r}}^{\text{ref}}(t)\|_2\leq 0.5$ for $t\in[3,6)$, and its color is drawn in increasingly darker shades of grey as it is computed. Also shown is the speed of the Crazyflie2.1 as it tracks the evolving reference trajectory. ]
    {\includegraphics[width=8cm]{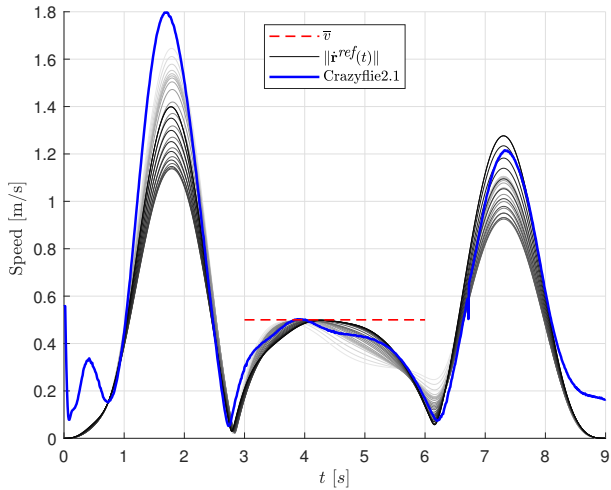}}
    \caption{
    Visualization of the experimental data for Example \ref{example:slowzone}.
    }
    
    \label{fig:slowzone}
\end{figure}

\end{example}

\begin{example} \label{example:forest}
The third example demonstrates safe trajectory planning and tracking in a cluttered environment. The definitions of the start/end zones $\mathcal{C}_1,\ldots,\mathcal{C}_4$, cuboid obstacles $\mathcal{O}_1,\mathcal{O}_2$, \emph{hoop} $\mathcal{O}_{3}$, \emph{desk} $\mathcal{O}_{4}$, \emph{ladder} $\mathcal{O}_{5}$, safe flight space $\mathcal{F}$, and the convex sets  $\mathcal{S}_1,\ldots,\mathcal{S}_6$ (chosen as in Proposition \ref{convex_sets_prop}) are described in Appendix \ref{appendix:forest}. The environment configuration is shown in  Figure \ref{fig:forest}.  We design trajectories that avoid the obstacles and take the quadcopter from any $\vect{p}^0\in\mathcal{C}_i$ to any  $\vect{p}^f\in\mathcal{C}_j$ such that:
\begin{align*} 
    & \vect{r}^{\text{ref}}(t) \in\mathcal{F},\  t\in[t_0,t_f],\\ & \vect{r}^{\text{ref}}(t_0) = \vect{p}^0 , \quad \vect{r}^{\text{ref}}(t_f) = \vect{p}^f. 
\end{align*}
Because of the choice of convex sets $\{\mathcal{S}_j\}$, given start zone $\mathcal{C}_i$  and end zone $\mathcal{C}_j$, we can always find a finite sequence ($n_{s}<\infty$) of indices $(k_1,k_2,\ldots,k_{n_s}), \ k_l\in\{1,\ldots,6\}, \ l=1,\ldots,n_s$ such that
\begin{align*}
    &(\mathcal{C}_i \cup \mathcal{S}_{k_1} \cup \mathcal{S}_{k_2} \cup \ldots \cup \mathcal{S}_{k_{n_s}} \cup \mathcal{C}_j)\in\mathcal{F},\\
    &\mathcal{C}_i \cap \mathcal{S}_{k_1} \neq \emptyset, \quad \mathcal{S}_{k_{n_s}}\cap\mathcal{C}_j\neq \emptyset,\\
    &\mathcal{S}_{k_l}\cap \mathcal{S}_{k_{l+1}}\neq \emptyset , \quad l = 1,\ldots,n_s-1,
\end{align*}
where $i,j\in\{1,\ldots,4\}$ are the indices of the start zone $\mathcal{C}_i$ and the end zone $\mathcal{C}_j$, respectively. The sequence of sets $(\mathcal{S}_{k_1},\ldots,\mathcal{S}_{k_{n_s}})$ satisfies the conditions of Proposition \ref{convex_sets_prop} and we can use it to obtain constraints for \eqref{FLAT-SOCP} such that its solution results in a B-spline curve $\vect{r}^{\text{ref}}(t)$ that safely navigates the cluttered environment.


In the experiment, we set $\vect{p}^{0}\in\mathcal{C}_i$ as the current position of the quadcopter, and choose a random point $\vect{p}^{f}\in\mathcal{C}_j$ as the next goal position where $\mathcal{C}_j\in\{1,\ldots,4\}$ is also randomly picked. Then, we find a sequence of convex sets $(\mathcal{S}_{k_1},\ldots,\mathcal{S}_{k_{n_s}})$ as above and obtain constraints for \eqref{FLAT-SOCP} from Proposition \ref{convex_sets_prop}. Finally, we solve the optimization problem \eqref{FLAT-SOCP} and generate a reference trajectory. When the quadcopter finishes tracking the trajectory using the CBF-QP-based controller, the process is repeated for a new random $\vect{p}^{f}\in\mathcal{C}_k,k\in\{1,\ldots,4\}$. We repeat this process 6 times, computing each trajectory on-the-fly. The result is a smooth, continuous trajectory that safely navigates the cluttered environment, shown in Figure \ref{fig:forest}. In the experiment, we forced the last $\vect{p}^f$ to match the takeoff location of the quadcopter. The resulting sequence of points visited by the quadcopter was $\vect{p}_1\in\mathcal{C}_2, \vect{p}_2\in\mathcal{C}_3, \vect{p}_3\in\mathcal{C}_2,\vect{p}_4\in\mathcal{C}_1, \vect{p}_5\in\mathcal{C}_3, \vect{p}_6\in\mathcal{C}_4, \vect{p}_1\in\mathcal{C}_2$, and  $\vect{p}_l,l=1,\ldots,6$ are shown in Figure \ref{fig:forest}.

\begin{figure}
    \centering
    \includegraphics[width=7.5cm]{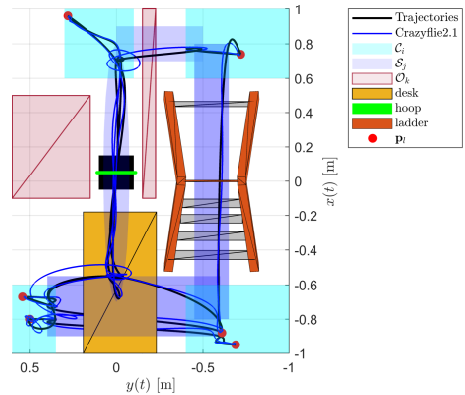}
    \includegraphics[width=7.5cm]{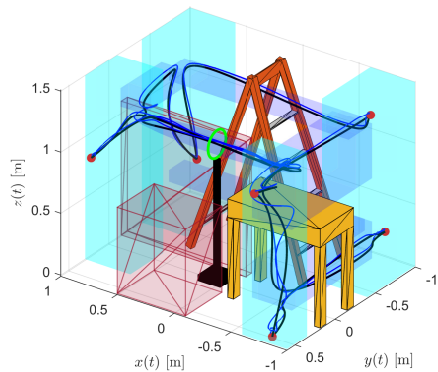}
    \caption{Experimental setup of the cluttered environment and the trajectories generated  for Example \ref{example:forest}. The designated takeoff/landing position is $\vect{p}_1\in\mathcal{C}_2$. The other points $\vect{p}_l,l=2,\ldots,6$ are chosen randomly from the zones $\mathcal{C}_i,i=1,\ldots,4$. The sequence of points visited by the quadcopter is $\vect{p}_1\in\mathcal{C}_2, \vect{p}_2\in\mathcal{C}_3, \vect{p}_3\in\mathcal{C}_2,\vect{p}_4\in\mathcal{C}_1, \vect{p}_5\in\mathcal{C}_3, \vect{p}_6\in\mathcal{C}_4, \vect{p}_1\in\mathcal{C}_2$. The reference trajectories are shown in black and the real trajectory of the recorded tracking data of the Crazyflie 2.1 quadcopter is shown in blue.  We emphasize that the trajectories were generated on-the-fly as each end point was drawn after completing the tracking of the previous trajectory. The average solve-time of \eqref{FLAT-SOCP} was 0.0135 seconds (see Remark \ref{online_replaning_remk}). The supplementary video is given in \cite{victor21video}.} 
    \label{fig:forest}
\end{figure}
\end{example}

\begin{remark} \label{online_replaning_remk}
For experiments that require online replanning, we use Yalmip's \texttt{optimizer} objects to reduce the overhead compiling time, which greatly improves solve time of parametric problems. Typical compile times for an \texttt{optimizer} object of \eqref{FLAT-SOCP} are about $0.7$ seconds and, as previously stated, any subsequent solving of the parameterized \eqref{FLAT-SOCP} takes between $0.01$ and $0.033$ seconds. Note also that even though the segment-wise constraints of Proposition \ref{convex_sets_prop} are not easily parameterized, we can compile multiple \texttt{optimizer} objects that cover all possible scenarios offline. Then, replanning can still be done online (and easily parallelized) for unstructured scenarios as demonstrated in Example \ref{example:forest}. 
\end{remark}

\section{Conclusions} \label{sec:Conclusions}
We presented a framework based on convex optimization that achieves rigorous continuous-time safety for the planning and tracking of quadcopter trajectories. We also showed several experimental examples that demonstrate the usefulness of the presented framework. We highlight that the presented trajectory optimization method can easily generalize to other differentially flat systems, especially when geometric interpretation of the flat-space constraints can be inferred. The presented safe tracking approach can also generalize to other control-affine mechanical systems. In future work, we will explore customized solvers for the SOCP problem involved, and incorporate high-level ``safe corridor'' finding algorithms into the proposed framework.


\bibliographystyle{IEEEtran}
\bibliography{IEEEabrv,trajbib}

\appendix

\subsection{Construction of Matrix $B_r$}\label{B_mat}
The matrix $B_r\in\R^{(N+1)\times(N+r+1)}$ is only defined when $0\leq r\leq d$, and can be computed from two matrices $M_{d,d-r}$ and $C_r$ such that $B_r = M_{d,d-r}C_r$; see \cite{fajar_thesis} for more details. 

The matrix $M_{d,d-r}\in\R^{(N+1)\times(N-r+1)}$ can be constructed recursively by:
\begin{align*}
    M_{d,d-r} &= \begin{cases} I_{N+1},  & r= 0,\\
    \prod_{i=1}^rf_M(\vect{\tau},d,N,i), & 1\leq r \leq d,
    \end{cases}
\end{align*}
where each iteration in the product is a right-multiplication and the matrix-valued function $f_M: \mathbb{R}^{v+1}\times \mathbb{Z}_{>0}\times\mathbb{N}\times\mathbb{N} \rightarrow \mathbb{R}^{(N-i+2)\times(N-i+1)}$ is defined as:
\begin{align*}
    f_M(\vect{\tau},d,N,i) 
    & = \mat{-a_0 &\dots & 0\\
              a_0 & \ddots & \vdots\\
              \vdots &\ddots&-a_{N-i}\\
              0 & \hdots & a_{N-i}},\\
    a_k &= \frac{d-i+1}{\tau_{k+d+1}-\tau_{k+i}}.
\end{align*}
 
The matrix $C_r\in\R^{(N-r+1)\times(N+r+1)}$ is constructed as
    $$C_r = \mat{\vect{0}_{(N+1-r)\times r}& I_{N+1-r}& \vect{0}_{(N+1-r)\times r}}
    $$ 
for $r\in \mathbb{Z}_{>0}$ and $C_0=I_{N+1}$.


\subsection{Configuration of the Environment for Example \ref{example:forest}}\label{appendix:forest}
The start/end zones $\mathcal{C}_1,\mathcal{C}_2,\mathcal{C}_3,\mathcal{C}_4$ are described by 
\begin{align*}
    \mathcal{C}_i &\triangleq \{\vect{r}\in\mathbb{R}^3\mid A\vect{r}\leq \vect{b}^c_i\},\quad i=1,2,3,4,\\
    \mbox{where}\;\quad A &= \mat{I_3 & -I_3}^T,\\
    \mat{{\vect{b}^c_1}^T\\{\vect{b}^c_2}^T\\{\vect{b}^c_3}^T\\ {\vect{b}^{c}_4}^T} &= \mat{1&0.3&1.5&-0.6&0.1&0\\ -0.6&0.6&1.5&1&-0.35&0\\
    -0.6&-0.4&1.5&1&1&0\\1&-0.4&1.5&-0.6&1&0
    }.
\end{align*}
The cuboid obstacles $\mathcal{O}_1,\mathcal{O}_2$ are described by:
\begin{align*}
    \mathcal{O}_i &\triangleq \{\vect{r}\in\mathbb{R}^3\mid A\vect{r}\leq \vect{b}^o_i\},\;i=1,2,
    \end{align*}
    where $A$ is the same as above, and 
    \begin{align*}
    \mat{{\vect{b}^o_1}^T\\{\vect{b}^o_2}^T} &= \mat{0.5&0.6&0.7&0.1&-0.15&0\\ 0.1&-0.225&1.1&0.1&0.15&0}.
\end{align*}
The other obstacles $\mathcal{O}_3,\mathcal{O}_4,\mathcal{O}_5$ are:
\begin{itemize}
    \item $\mathcal{O}_3$: a \emph{hoop} centered at $(0, 0,1.1)$ through which the quadcopter can pass.
    \item $\mathcal{O}_4$: a \emph{desk} between $\mathcal{C}_2$ and $\mathcal{C}_3$ so that the quadcopter must fly above or below the \emph{desk} to avoid it.
    \item $\mathcal{O}_5$: a \emph{ladder} between $\mathcal{C}_3$ and $\mathcal{C}_4$ so that the quadcopter must fly between two of its rungs.
\end{itemize}
The considered obstacles allow us to define the safe flight space:
\begin{align*}
    \mathcal{F} &= \mathcal{B} \setminus( \mathcal{O}_1 \cup\cdots\cup \mathcal{O}_5),\\
    \mathcal{B}&\triangleq \{\vect{r}\in\mathbb{R}^3\mid \underline{\vect{b}}\leq\vect{r}\leq \overline{\vect{b}}\},\\
    \underline{\vect{b}} &= \mat{-1 & -1 & 0}^T, \quad \overline{\vect{b}} = \mat{1 & 0.6 & 1.5}.
\end{align*}
Note that $\mathcal{C}_i\in\mathcal{F}$ for $i=1,\ldots,4$. We choose 6 convex sets $\mathcal{S}_j\in\mathcal{F},\ j=1,\ldots,6$, defined as:
\begin{align*}
    \mathcal{S}_j &\triangleq \{\vect{r}\in\mathbb{R}^3\mid A_j\vect{r}\leq \vect{b}^s_j\}, \quad j = 1,\ldots,5,
\end{align*}
where $A_1,\ldots,A_5=A$ and 
\begin{align*}
    \mat{{\vect{b}^s_1}^T\\{\vect{b}^s_2}^T\\{\vect{b}^s_3}^T\\{\vect{b}^s_4}^T\\{\vect{b}^s_5}^T} &= \mat{
    -0.55&0.4&0.5&0.9&0.6&-0.2\\
    -0.55&0.4&1.3&0.9&0.6&-0.85\\
    0.8&-0.45&1.3&0.8&0.65&-1.2\\ 0.8&-0.45&0.47&0.8&0.65&-0.37\\ 0.9&0&1.3&-0.7&0.5&-1.15},\\
    \mathcal{S}_6 &\triangleq \{\vect{r}\in\mathbb{R}^3:\|A^s_6\vect{r}+\vect{b}_6^s\|\leq 1\},\\
    A_6^s &= \text{diag}(1.33,13.3,13.3),\\
    \vect{b}_6^s &= \mat{-0.067 & 0 & -14.67}^T.
\end{align*}

%

\end{document}